\documentclass[letterpaper,12pt]{article}
\usepackage[utf8]{inputenc}
\usepackage[inner=2.5cm,outer=2.5cm]{geometry}
\usepackage[english,activeacute]{babel}
\usepackage{amssymb,amsmath,alltt,amsthm}
\usepackage{amsfonts}
\usepackage{enumitem}
\usepackage{graphicx}
\usepackage{authblk}
\usepackage{soul}
\usepackage{hyperref}

\usepackage[title]{appendix}
\usepackage{xcolor,colortbl}

\definecolor{Gray}{rgb}{0.91,0.91,0.91}

\usepackage{ctable}
\usepackage{multirow}
\usepackage{array}
\newcolumntype{L}[1]{>{\raggedright\let\newline\\\arraybackslash\hspace{0pt}}m{#1}}
\newcolumntype{C}[1]{>{\centering\let\newline\\\arraybackslash\hspace{0pt}}m{#1}}
\newcolumntype{a}{>{\columncolor{Gray}}c}

\definecolor{red}{RGB}{250,0,0}
\definecolor{purple2}{RGB}{120,0,120}
\definecolor{blue2}{RGB}{0,0,180}
\definecolor{green2}{RGB}{0,180,0}
\definecolor{grey2}{gray}{0.9}
\definecolor{grey3}{gray}{0.95}

\usepackage[section]{placeins} 

\newcommand{\R}{\mathbb{R}}

\newcommand{\curl}{\textrm{curl}}
\newcommand{\eps}{\varepsilon}
\renewcommand{\div}{\textrm{div}}

\newcommand{\calV}{\mathcal{V}}
\newcommand{\calT}{\mathcal{T}}

\newtheorem{thm}{Theorem}
\newtheorem{lemma}{Lemma}
\newtheorem{innercustomthm}{Theorem}
\newenvironment{customthm}[1]
  {\renewcommand\theinnercustomthm{#1}\innercustomthm}
  {\endinnercustomthm}

\title{Numerical tests for volume enclosed by flux surfaces in toroidal geometry}

\author[1]{D.Martínez-del-Río\footnote{david.martinezdelrio.math@gmail.com}}
\author[1]{R.S.MacKay}
\affil[1]{Mathematics Institute, University of Warwick, Coventry CV4 7AL, UK}

\date{}

\begin{document}

\maketitle

\begin{abstract}
Numerical tests of volume formulae are presented to efficiently compute  the volume enclosed between flux surfaces for integrable 3D vector fields with various degrees of {symmetry}. In the process, a new case is proposed and tested.
\end{abstract}


\section{Introduction}
\label{sec:intr}

In the design of magnetic confinement devices for plasma, it is often important to determine the volume $V$ enclosed by a flux surface $S$, particularly the volume bounded by the outermost flux surface of a given class within the vacuum vessel. This is one objective function that is considered in computational optimisation of designs, e.g.,~\cite{SIMSOPT}. Given the field, the~volume is usually found by simple 3D integration. 

An  approach presented in~\cite{mackay2024volume} derives expressions for the volume that take advantage of the invariance of $S$ under the magnetic field $B$ for a range of integrable cases, thereby simplifying the conventional three-dimensional integration and potentially enabling more efficient~computation.

The purpose of this work is to test the volume formulae in~\cite{mackay2024volume} numerically on relevant, non-trivial examples, in~order to demonstrate the advantages or potential issues arising from their~use. 

The {explicit} example fields available for numerical testing are rather limited, as~these formulae require {some form of integrability, in~particular} the existence of a {continuous family of flux surfaces and an associated} label function $\Psi$, so we restrict our attention to example fields that are integrable.  This is, however, a~common assumption in the plasma physics context, e.g.,~\cite{helander2014theory}, because~it holds for all non-degenerate magnetohydrostatic (MHS) fields ({\em {MHS} 
} means $J\times B = \nabla p$ for some function $p$, where $J = \curl\, B$, and~{\em {non-degenerate}} 
means $\nabla p \ne 0$ except on a set of measure zero). It is also a common assumption in numerical codes for computing MHS fields and more generally MHD equilibria (but see Ref.~\cite{HHBB} for one that does not make this assumption and presents references to various ones that do.)

Section~\ref{sec:Volume_formulas} provides a brief exposition of the terminology and results from~\cite{mackay2024volume}, plus a new result that treats the case of integrable magnetic fields for which the associated symmetry preserves some density that is not necessarily constant. Section~\ref{sec:magfields} describes the analytical magnetic fields used to test these results. Section~\ref{sec:results} presents our numerical findings, which are discussed in Section~\ref{sec:discussion}, and~the conclusions are summarised in Section~\ref{sec:conc}.

\section{Volume enclosed by flux surfaces}
\label{sec:Volume_formulas}

A magnetic field is a $C^1$ vector field $B$ on an orientable three-dimensional manifold $M$ that preserves a volume form $\Omega$ (for the standard Euclidean volume, this is the condition $\mathrm{div}\, B = 0$). 

We will formulate most of our exposition in the language of exterior calculus because it makes many concepts much simpler to express and some results more evident. It is less familiar to plasma physicists, however, to~whom this work is particularly directed, so we translate into vector calculus language in many~places.

In the language of exterior calculus, to~a magnetic field $B$, an~associated flux 2-form can be  defined as $\beta = i_B \Omega$, allowing the volume-preservation condition to be written as $d\beta = 0$ (i.e., $\beta$ is closed:~its integral over any closed surface contractible to a point is zero). The~vector calculus interpretation of $\beta$ is that acting on a pair of vectors $\xi$ and $\eta$, $\beta(\xi, \eta) = B \cdot dS$ for the area element $dS$ given by the oriented parallelogram spanned by $\xi$ and $\eta$.  So $\beta$ can be integrated over a surface $S$ by chopping $S$ into little parallelograms, summing their contributions and taking the limit.  The~exterior derivative $d\beta$ of $\beta$ is defined to be the 3-form, acting on oriented triples of vectors, that gives the sum of $\beta$ over the faces of the parallelepiped they span, oriented outwards.  The~exterior derivative $d$ on 1-forms is defined similarly, and~on functions $f: \R^3 \to \R$, $df$ is just the ordinary derivative.
A tutorial on the use of differential forms in  plasma physics can be found in~\cite{mackay2020differential}.

The fields relevant to the present work, and~to much of plasma physics, typically assume  the slightly stronger condition that $\beta$ is exact, i.e.,\ $\beta = d\alpha$ for some 1-form $\alpha$. This is equivalent to requiring the existence of a vector potential $A$ such that $B = \nabla \times A$, where $\alpha = A^\flat$, {the 1-form that when acting on any vector $\xi$ produces $A^\flat(\xi) = A \cdot \xi$. Another way to define exactness of $\beta$ is} that $\int_S \beta = 0$ for any closed surface $S$, not only  contractible~ones.

An invariant torus of $B$ is called a \emph{flux surface}.
A magnetic field is called {\em {integrable}}.
 if there is a $C^1$ function $\Psi$ {(called a {\em {flux function}}
 or flux surface label)} such that $d\Psi \ne 0$ almost everywhere (a.e.) and $i_B d\Psi = 0$ (in vector calculus, $\nabla \Psi \ne 0$ a.e.~and $B \cdot \nabla \Psi = 0$, but~the property does not require the Riemannian metric implicit in $\nabla$ and $\cdot$, which is why we prefer to express it using differential forms). For~an integrable field, the~level sets of $\Psi$ are invariant under $B$.  If~$B$ and $d\Psi$ are nowhere-zero on a compact component of a level set, then it is a torus and hence a flux surface.  Furthermore, it has a neighbourhood that is foliated by flux surfaces.
From KAM theory, if~$B$ is non-integrable but smooth enough, then various simple conditions imply a set of positive volumes of flux surfaces, but we will not be using that~here.

To quantify the volume $V$ enclosed by a flux surface $S$, once the surface has been computed, the~enclosed volume can be obtained via a 3D integral $\int_V \Omega$. This can be reduced to a 2D integral $\int_S \nu$ {for some 2-form $\nu$} if the manifold $M$ is contractible (i.e.,\ $\Omega = d\nu$ for some 2-form $\nu$ {and then the result follows by Stokes' theorem}). {For example, if~$\Omega = dx \wedge dy \wedge dz$, as~for a standard volume in Cartesian coordinates, then one can take $\nu = z\, dx \wedge dy$, and integrating this over a closed surface corresponds to integrating the height difference between the points of the surface above a given point $(x,y)$ of the horizontal plane, with~respect to horizontal area.}  However, this direct approach does not take into account the invariance of $S$ under $B$. 

The first formula presented in~\cite{mackay2024volume} for computing the volume that exploits this {invariance is:} 
\begin{equation}
\label{V_general}
V = \int_D \mathsf{T} \beta \,,
\end{equation}
where $D$ is a disk transverse to $B$ whose boundary is on the flux surface $S$, and~$\mathsf{T}$ is the first-return time to $D$ along the fieldline flow $\dot{x}=B(x)$. Although~(\ref{V_general}) is a 2D integral, it involves computing the return-time function, which effectively makes the integration 3D, so there is no real computational reduction 
(in the Appendix, however, a~way to reduce the computation of $V$ from (\ref{V_general}) as a function of $\Psi$ for integrable fields is given).

So we are restricted to a sequence of  special cases of magnetic field, for~which the integration can be reduced.  In~the two strongest cases, it is reduced to 2D; in the other two cases, it is reduced to a bit more than 2D, in~a sense to be~described. 

All cases considered are integrable, but~they differ in their degree of symmetry.
To explain this, note that if $B$ is integrable and non-zero a.e.,~then there is a vector field $u$ independent of $B$ a.e.~such that $L_u\beta=0$, where $L_u$ denotes the Lie derivative along $u$, that is, $u$ is a symmetry of $\beta$.  The~expression of this in vector calculus language is $\curl (B \times u) + (\div B) \, u= 0$.
For example, if~$B$ is integrable (and divergence-free, always assumed), with~flux function $\Psi$, one can take $u = B\times \nabla \Psi/|B|^2 + f B$ using any Riemannian metric and any function $f$.  
But there might be values of $u$ which preserve more than just $\beta$.  For~example, $u$ might also preserve $\rho \Omega$ ($L_u (\rho \Omega)=0$, equivalently $\div(\rho u) = 0$) for some positive function $\rho$, which we call a density.  Or~it might preserve volume $\Omega$ (so constant density, $\div\, u = 0$); in~which case,  $L_u B$  is also zero (it is equal to the commutator $[u,B]$, which can be written as $\curl (B \times u) - \div u \, B + \div B \, u$).  If~in addition $u$ preserves $|B|$ ($L_u |B|=0$, equivalently, $u \cdot \nabla |B|=0$), then we say $u$ is a {\em {weak quasisymmetry}}
 of $B$, or it might furthermore preserve $B^\flat$ (defined the same way as $A^\flat$), at which point we say $u$ is a {\em {quasisymmetry}}

  of $B$ (in vector calculus, $L_u B^\flat=0$ is written as $(\curl B) \times u + \nabla(u\cdot B) = 0$).

Note that for an integrable field, it makes sense to ask to compute not just the volume enclosed by one flux surface, but~instead the whole function giving the volume $V(\Psi)$ enclosed as a function of $\Psi$. Our methods will compute this by finding $dV/d\Psi$ as a 1D integral (or a bit more) and then the function $V(\Psi)$ by one more~integration.

The proof of (\ref{V_general}), as~well as those of  Theorems \ref{thm1}--\ref{thm4}, can be found in~\cite{mackay2024volume}.  A~subsection on a new result, Theorem \ref{thm3'}, and~its proof, is~included.

In contrast to what we have just described, we will now treat the cases in order of decreasing~specialisation.

\subsection{Quasisymmetric fields}
\label{subsec:quasisym}
A \emph{{quasisymmetric}}
 {(QS)} magnetic field $B$ is the special case in which {there} exists a vector field $u$ independent from $B$ almost everywhere (in the region of interest), with~\begin{equation}
\label{quasisymm_cond}
L_u \beta=0\,, \qquad
L_u \Omega =0\,, \qquad  
L_u B^\flat=0 \, .
\end{equation}
{Expressions} vector calculus language were given above.
Under mild additional conditions~\cite{burby2020some}, every orbit of $u$ is closed and has the same period $\tau >0$. 

The only known analytic examples of these fields {(with bounded $u$ orbits)} are axisymmetric, where $u= \partial_\phi$ in cylindrical coordinates (the vector field with $\dot{\phi}=1, \dot{r}=0, \dot{z}=0$) (so $\tau = 2\pi$). It is an open question whether there are any other exactly quasisymmetric fields, although~fields with non-axisymmetric quasisymmetry can be constructed to a high degree of accuracy~\cite{LP22}.

The conditions $L_u \beta = 0$ and $L_u \Omega = 0$ imply that $i_u \beta$ is a closed form, i.e.,\ $\int_\gamma B \times u \cdot dl$, and, along a path, $\gamma$ is unchanged under continuous deformation of the path preserving the ends. Assuming there is no homological obstruction (one might want to assume that every loop is contractible to a point but~is stronger than necessary and rules out toroidal geometry; it is enough to suppose that each homology class in dimension 1 can be realised by an orbit of some linear combination of $u$ and $B$), this further implies that $i_u \beta$ is exact; i.e.,~there exists a function $\Psi$ such that $i_u \beta = d\Psi$ ($B\times u = \nabla \Psi$). 
 It follows that both $u$ and $B$ are tangent to the level sets of $\Psi$,
$\Psi$ is a \emph{flux} or a \emph{label function}, 
  and~that the bounded components of the level sets are~tori.

Using the volume-preserving property and the conditions in (\ref{quasisymm_cond}), ref.~\cite{mackay2024volume} provides the following formula. Choose a $u$-line $\gamma$ and a point $x \in \gamma$, then
\begin{thm}\label{thm1}
$dV = \tau\, T(\Psi)\, d\Psi$, {where} the {return time} $T>0$ is the first time for which the flow of $B$ starting at $x$ returns to $\gamma$.
\end{thm}

\noindent The return time can be proved to depend on $x$ through only $\Psi(x)$.
As~computing the return time $T(\Psi)$ requires a 1D integration, the~entire volume calculation is thereby reduced to a 2D integral. For~clarity, the~return time $T(\Psi)$ corresponds to the time between two successive crossings of a fieldline on $S$ with a $u$-line $\gamma$, while $\mathsf{T}$ in (\ref{V_general}) corresponds to the time between two crossings through a given transverse disk $D$.

This result generalises to the case of a magnetic field with \emph{{weak quasisymmetry}}
  \cite{RHB20}; this is defined by requiring
\begin{equation}
\label{w_quasisymm_cond}
L_u \beta=0\,, \qquad
L_u \Omega =0\,, \qquad  
L_u |B| =0 \,.
\end{equation} 
The only relevant consequence for the volume formula is that now the value of $\tau$ is in general a function of $\Psi$. The~extended result is the following:
\begin{thm}\label{thm2}
$dV = \tau(\Psi)\, T(\Psi)\, d\Psi$.
\end{thm} 

\noindent {Thus,} 
 the~computation of the volume is still equivalent to a 2D integration.  It is an open question again, however, whether there are any weak QS fields that are not~axisymmetric.

\subsection{Fields with Flux-Form~Symmetry}
\label{subsec:flux-form_sym}
The next case considered in~\cite{mackay2024volume} is magnetic fields $B$ with \emph{{flux-form symmetry}}
, defined as having a volume-preserving {vector field} $u$ independent from $B$ almost everywhere, such that $L_u \beta = 0$ (recall, this means $\curl(B\times u)=0$ because we are assuming $\div B = 0$). Perhaps the name is insufficiently clear, because~for the definition it is essential that $u$ be volume-preserving, not just $L_u\beta=0$, but~we continue to use it.
It is equivalent to require $u$ to be a volume-preserving field that commutes with $B$ (e.g.,~\cite{burby2020some}).  It holds for all non-degenerate MHS fields that~$u = J$, though~all examples that we know of such fields are axisymmetric (so are already covered by {Theorem} 
 \ref{thm1}).

Under the same homological condition, this implies again that $i_u\beta = d\Psi$ for some function $\Psi$ whose level sets are invariant under both $u$ and $B$. {Furthermore, the~commutation condition implies that $(u, B)$ generates an action $\varphi$ of $\R^2$ on the flux surfaces. For~each pair $(t_1,t_2) \in \R^2$ and point $x$ on a flux surface, $\phi_{(t_1,t_2)}x$ is given by flowing from $x$ for time $t_1$ with $u$ and time $t_2$ with $B$, in~any order because the two fields commute.  This is a place where formulation in vector calculus makes it harder to reach the result, though~it was achieved by Hamada~\cite{Ha62}.  
From the condition $\curl(B\times u) - \div u\, B + \div B\, u = 0$, one has to prove that flowing along $u$ for a time $t_1$ and along $B$ for time $t_2$ commutes.}
Consequently, for~each flux surface, there exists a lattice $\Gamma \subset \R^2$ of pairs $t \in \R^2$ of times such that $\varphi_t = \mathrm{Id}$ 
if and only if $t \in \Gamma$. Let ${T}_1, {T}_2 \in \mathbb{R}^2$ be generators of $\Gamma$. Form a matrix $\mathcal{T}$ with the vectors ${T}_1$ and ${T}_2$ as columns, and~let $\Delta = \det \mathcal{T}$, which is a function of $\Psi$. The~volume formula for this case can then be written as
\begin{thm}\label{thm3}
$dV = \Delta(\Psi)\, d\Psi$.
\end{thm} 

\subsection{Integrable Fields with Symmetry Preserving a~Density}
\label{subsec:int_fields_dens}
In preparing this paper, we realised that several examples, including one that we treat numerically, satisfy $L_u\beta=0$, but with $u$ preserving a non-constant density $\rho$, i.e.,~$L_u (\rho \Omega)=0$ for some smooth positive function $\rho$.  Examples of this case that are not covered by a preceding Theorem are given by the fields with helical symmetry from~\cite{KMM23}, which is the example we will treat here, and~MHD equilibria with flow and electrostatic potential.  In~the latter example, the~electrostatic potential provides an integral; the symmetry field is the plasma velocity, which in general is not volume-preserving but does preserve plasma~mass.

Although this case can be treated by Theorem \ref{thm4} below, we have derived a potentially more efficient formula for it that we explain~here.

{As before, $L_u\beta=0$ plus a homological condition implies there is a flux function $\Psi$. The~next step is to notice that $u$ commutes with $B/\rho$.}

\begin{lemma}
\label{lemma1}
$L_u\beta=0$ and $L_u(\rho \Omega)=0$ imply $[u,B/\rho]=0$.
\end{lemma}

 \noindent {In}
 vector calculus language, this says that $\curl(B\times u) + \div B\, u=0$ and $\div(\rho u)=0$ (together with $\div B=0$ assumed throughout) implies $\curl(\frac{B}{\rho}\times u) - \div u\, \frac{B}{\rho} + \div\frac{B}{\rho}\, u = 0$.  A~vector calculus reader may prove this for themselves.  We give an exterior calculus proof.
\begin{proof}
A standard result for commutators of vector fields yields
\begin{equation}
i_{[u,B/\rho]}\Omega = L_u i_{B/\rho}\Omega - i_{B/\rho} L_u\Omega .
\label{eq:comm}
\end{equation}
Recalling that $i_B \Omega = \beta$, the~first term gives
$$\frac{1}{\rho} L_u\beta - \frac{L_u\rho}{\rho^2} \beta  = - \frac{L_u\rho}{\rho^2} \beta ,$$
using $L_u\beta=0$.
Next, use $L_u(\rho\Omega)=0$ to obtain $L_u\Omega = -\frac{L_u\rho}{\rho} \Omega$.  Then the second term of (\ref{eq:comm}) gives $-\frac{L_u\rho}{\rho^2}\beta$, which cancels the first.  As~$\Omega$ is non-degenerate, $i_{[u,B/\rho]}\Omega=0$ implies that $[u,B/\rho]=0$. 
\end{proof}

Together with the linear independence of $u$ and $B$, this implies that  on any flux surface $S$,  the pair $(u, B/\rho)$ of vector fields induces an action $\phi$ of $\R^2$, and~the set of pairs $t \in \R^2$ of times for which the resulting $\phi_t$ is the identity forms a lattice $\Gamma \subset \R^2$.  As~for Theorem~\ref{thm3}, let $T_1, T_2 \in \R^2$ be generators for $\Gamma$ and $\Delta$ be the determinant of the matrix $\calT$ they form. Furthermore, the~action $\phi$ allows one to parametrise the flux surface by a pair of angle variables $(\theta_1, \theta_2)$ modulo 1, which evolve at constant speeds with respect to the $u$ flow and the $B/\rho$ flow.  They are called {\em {Arnol'd--Liouville}} 
 (AL) coordinates~\cite{A}.

To relate the volume enclosed by $S$ to its value of $\psi$, one more ingredient is required, namely the {\em {harmonic average}} 
 $\hat{\rho}$ of $\rho$ on $S$, with~respect to the AL coordinates $(\theta_1,\theta_2)$ on it:
$$\frac{1}{\hat{\rho}} = \int_S \frac{1}{\rho}\, d\theta_1 \wedge d\theta_2.$$
This can be computed as the limit of the average of $1/\rho$ at the points of a regular subdivision of the lattice $\Gamma$; i.e.,~for a large enough integer $q$, choose a starting point on $S$, flow with $(u,B/\rho)$ for vector-times $(n_1 T_1 + n_2 T_2)/q$ with integers $n_1, n_2 \in \{0, \ldots q-1\}$, evaluate $1/\rho$ there, and~take the average.  If~everything is analytic, then by Fourier analysis the error is exponentially small in $q$ (see~\cite{weideman2002numerical} 
for an entertaining exposition of this), so in practice one can expect to get away with a relatively small $q$.  In~this sense, we think of evaluating $\hat{\rho}$ as only a bit more than a 1D~integral.

We obtain the following theorem for the change in the volume enclosed by a flux surface for a change in the value of $\Psi$, numbered 3' because its hypotheses are intermediate between those for Theorems 3 and 4 {of}~\cite{mackay2024volume}.

\begin{customthm}{3'} \label{thm3'}
{If} 
 $L_u\beta=0$ and $L_u(\rho\Omega)=0$, then $$dV = \frac{\Delta}{\hat{\rho}} d\Psi .$$
\end{customthm} 

\begin{proof}
Recall that in AL coordinates $(\theta_1,\theta_2)$, the~vector fields $u$ and $B/\rho$ have constant components on each flux surface.
Then our key step is to notice that
\begin{equation}
\Omega = \frac{\Delta}{\rho} d\theta_1 \wedge d\theta_2 \wedge d\Psi.
\label{eq:key}
\end{equation}
To prove this, let $n$ be a vector field such that $i_n d\Psi=1$ (equivalently, $n\cdot \nabla \Psi = 1$), e.g.,~$n = \nabla \Psi/|\nabla \Psi|^2$ (for an arbitrary Riemannian metric), and apply $i_u i_{B/\rho}i_n$ to both sides of (\ref{eq:key}).  Applied to $\Omega$ this gives $\frac{1}{\rho} i_u i_B i_n\Omega$ (this is $\frac{1}{\rho}$ times the triple product of $(n,B,u)$), which reduces to $\frac{1}{\rho}$ because $i_u i_B\Omega = d\Psi$ and $i_n d\Psi = 1$.
Applied to $d\theta_1 \wedge d\theta_2 \wedge d\Psi$, it gives $\det \calV$, where
$\calV$ is the matrix formed by the $(\theta_1,\theta_2)$ components of $u$ and $B/\rho$.  Here we have again used that $i_n d\Psi = 1$.  Now $\calV \calT$ is the identity matrix, and so 
$$i_u i_{B/\rho} i_n \frac{\Delta}{\rho} d\theta_1 \wedge d\theta_2 \wedge d\Psi = \frac{1}{\rho}.$$  Because the space of top-forms at a point is one-dimensional, the~two sides of (\ref{eq:key}) are~equal.

Then it follows by integration over $S$ that $dV = \frac{\Delta}{\hat{\rho}} d\Psi$. 
\end{proof}

Apart from the extra factor $\hat{\rho}$, using this formula is a simple modification of the case of Theorem \ref{thm3}.

\subsection{Integrable~Fields}

After {Theorem} 
 \ref{thm3}, ref.~\cite{mackay2024volume} presents a weaker version, in~which {$L_u\beta=0$ but} $u$ is not required to be volume-preserving, nor even to preserve a density, so it is weaker than Theorem \ref{thm3'} too. This case is called {\em {integrable}} 
  because it is equivalent to the existence of a flux function $\Psi$ whose level sets are invariant (modulo the homological condition).  All the previous cases are also integrable but with more conditions on the symmetry field $u$. 
The result~is

\begin{thm}\label{thm4}  
Choose a closed curve $\gamma$ on each flux surface, depending smoothly on $\Psi$.  Then
$dV = \bar{T}\, d\Phi$,
where 
\begin{equation*}
    \Phi = \int_\gamma A^\flat \,,
\end{equation*}
and $\bar{T}$ is defined as the limiting average return time to $\gamma$ of an infinitely long $B$-line starting at any point of $\gamma$. \end{thm}

\noindent {Note} 
 that $\Phi$ is the magnetic flux across a disk spanning $\gamma$; 
{ref.} .
 \cite{mackay2024volume} gives a way to compute it if the vector potential is not given explicitly.
One can compute $\frac{d\Phi}{d\Psi}$, so this is again in the form of evaluating $dV/d\Psi$.  We call it a bit more than a 1D integral because evaluating the average return time is longer than a bounded 1D~integral.

A question is whether there are integrable magnetic fields for which there is no choice of symmetry field $u$ that preserves a density. Integrability implies $B$ is conjugate to $\lambda C$ for some function $\lambda$ and constant vector $C$ for each flux surface~\cite{N},  
a conclusion that would follow if $\beta$ has a symmetry field $u$ that preserves $\lambda \Omega$, so this suggests that the answer is no.  Indeed, {ref.}
~\cite{PDP}  
shows that in any ``toroidal component'' for an integrable magnetic field, there is a choice of symmetry field that preserves a density.  Their construction uses a result of~\cite{PPS}, 
which in turn uses a result of Calibi that does not look explicit to us, but~they give explicit choices in particular contexts. For~example, if~$\curl\, B$ preserves the same function $\Psi$, then $u = B \times \nabla \Psi/|B|^2$ works, with~$\rho = |B|^2$.  As~another example, if~$B$ has a global cross-section, then Birkhoff proved there is an angle variable $\phi$ that increases by 1 for each return to the section, and~\cite{PDP} proved that there is a symmetry field $u$ preserving $\rho=L_B\phi$ (namely the vector field along the contours of $(\phi,\Psi)$ normalised to make $i_ui_B\Omega = d\Psi$). 
An interesting challenge is to work out the set of all compatible pairs $(u,\rho)$ for a general integrable $B$-field, including at least one explicit case.
A curious thing is that the construction of~\cite{PDP} always produces a symmetry field with a rational winding ratio, whereas there are axisymmetric cases for which this is not the case.  Indeed, one can add any flux function times $B/\rho$ to $u$ and get another symmetry preserving the same density, and this will make the winding ratio of $u$ change continuously, in~particular through irrational~values.

As a result of this discussion, Theorem \ref{thm4} may turn out to have little value.  Nevertheless, if~the construction of a compatible pair $(u,\rho)$ is awkward in an example, then it might still be useful.
So we also test it.


\section{Tested magnetic fields}
\label{sec:magfields}
The present work considers two different kinds of magnetic fields to test the volume formulae presented in the previous section. For~Theorem \ref{thm1}, we consider a simple axisymmetric tokamak field as in~\cite{burby2023isodrastic}. For~Theorems \ref{thm4} and \ref{thm3'}, we use a different integrable vector field consisting of an axisymmetric field perturbed by a single helical mode, described in~\cite{KMM23}. 

\subsection{Axisymmetric Magnetic~Field}
\label{subsec:axisymm}
As the only known quasisymmetric fields are the axisymmetric ones, our chosen example to test the quasisymmetric volume formula is axisymmetric.
Taken from Section~3.2 in~\cite{burby2023isodrastic}, an~example of an axisymmetric tokamak magnetic field in cylindrical polar coordinates $(R,\phi,z)$, is given by contravariant components
\begin{equation}
\label{axisym_B}
B^R = \frac{-z}{R}\,, \qquad
B^\phi = \frac{C}{R^2}\,, \qquad
B^z = \frac{r}{R} ,
\end{equation}
in a solid torus $r^2 + z^2 \leq r_0^2$ for some $r_0<1$, with~$r=R-1$, and~$C>0$. This field has a closed line for $z=0$, $R=1$,  {called} the `magnetic axis'. Corresponding to the symmetry $u = \partial_\phi$, the~fieldlines preserve a flux function,
\begin{equation}
\label{flux_func_axisymm}
\Psi = \frac{1}{2}\left( r^2 + z^2 \right).
\end{equation}

The flux surfaces for this field are two-tori of minor radius $\sqrt{2\Psi}$ and major radius $R_0 = 1$. Therefore the volume enclosed between a flux surface $\Psi>0$ and the magnetic axis ($\Psi=0$) is given by {a relatively simple integral that evaluates to}
\begin{equation}
\label{vol_axisym}
V= 
4 \pi^2 \Psi\,.
\end{equation}

\subsection{Toroidal Helical Magnetic~Fields}
\label{subsec:helical}
To test the volume formula of Theorem \ref{thm4} for  fields with {only integrability}, we opted to use the magnetic fields considered in~\cite{KMM23,MM25}, which correspond to perturbations of a circular tokamak field by helical modes, based on~\cite{kallinikos2014integrable}. 
Readers interested in  the details can  refer to~\cite{KMM23,MM25}.
As it turns out, to preserve a density, however, we take the opportunity to also test Theorem~\ref{thm3'}.

The fields considered have 
a circular magnetic axis of radius $R_0>0$ in the horizontal plane $z=0$.  They  are best described and treated in an adapted toroidal coordinate system $(\psi,\vartheta,\phi)$,  
which is a variant of the standard toroidal coordinates $(r,\theta,\phi)$. The~latter coordinates are related to Cartesian coordinates $(x,y,z)$ through 
$$x = R \sin\phi,\ y = R\cos \phi,\ z = r\sin \theta,$$
where
$$R= R_0 + r\cos\theta,$$
for  $0\le r<R_0$.  $R$ represents the cylindrical radius relative to the $z$-axis. In~these {standard} coordinates, the~metric tensor is represented by the matrix $\text{diag}(1,r^2,R^2)$.

Following~\cite{kallinikos2014integrable}, the~adapted coordinates $(\psi,\vartheta)$  are introduced to simplify the restriction of the magnetic flux-form $\beta$ to a poloidal section ($\phi=$ constant), and through some sensible choices (like setting $\psi$ as the toroidal magnetic flux across the poloidal disk of radius $r$ and selecting $B^\phi = {B_0R_0}/{R^2}$)~\cite{kallinikos2014integrable, abdullaev2014, KMM23}, we get
\begin{align}
\begin{split}
&\psi= B_0R_0^2\left(1-\sqrt{1-\frac{r^2}{R_0^2}}\right)\\
&\tan \frac{\vartheta}{2} = \sqrt{\frac{R_0-r}{R_0+r}}\tan\frac{\theta}{2}.
\end{split}
\end{align}

In terms of the covariant components of $A$, the~contravariant components of $B$ are given by
$$B^{i} = \frac{1}{\sqrt{|g|}} \epsilon^{ijk} \partial_j A_k,$$
where $\epsilon$ is the Levi-Civita symbol and $|g|$ is the determinant of the matrix $g$ representing the metric tensor, $ds^2 = g_{ij} dx^i dx^j$. In~our adapted toroidal coordinates, the~volume factor
\begin{equation}\sqrt{|g|} = 1/B^\phi=R^2/(B_0R_0).
\label{eq:volfactor}
\end{equation}

We take a vector potential, with one helical mode introduced in its toroidal component, of~the form {(in covariant components)}
\begin{align}
\label{eq:potential}
\begin{split}
A_\psi &= 0\\
A_\vartheta &= \psi \\
A_\phi &= -[w_1\psi+w_2\psi^2 + \eps \psi^{m/2} f(\psi) \cos(m\vartheta-n\phi + \zeta)],
\end{split}
\end{align}
where $w_1\in\R$, $w_2\ne 0$, $m,n$ are integers with $m\ge 2$, $f$ are smooth functions and $\zeta$ is an arbitrary~phase.

The vector potential (\ref{eq:potential}) gives rise to the magnetic field $B$, 
{whose} (contravariant) components are
\begin{align}
\label{eq:fields}
\begin{split}
B^\psi &= \frac{B_0R_0}{R^2}\left[m\eps\psi^{m/2} f(\psi)\sin(m\vartheta-n\phi + \zeta)\right] \\
B^\vartheta &= \frac{B_0R_0}{R^2} \left[w_1 + 2 w_2 \psi +  \eps \psi^{m/2-1} \left[\tfrac{m}{2}f(\psi)+\psi f'(\psi)\right]\cos(m\vartheta-n\phi + \zeta)\right]\\
B^\phi &= \frac{B_0R_0}{R^2} .
\end{split}
\end{align}
The cylindrical radius $R$ occurring in the conversion from $V$ to $B$ can be expressed in our adapted coordinates via
$$R = \frac{R_0^2-r^2}{R_0-r\cos\vartheta}$$
with $$r = R_0 \sqrt{1-\left(1-\frac{\psi}{B_0R_0^2}\right)^2}.$$

Because we take $m\ge 2$, all the fields have $\psi=0$ as a closed fieldline, which can be considered the {\em {magnetic axis}}.

As mentioned in~\cite{KMM23},  with~the addition of a single helical mode, the~field is still integrable~\cite{kallinikos2014integrable}.  Indeed, it has the invariant (i.e., integral of motion)
\begin{equation}
\Psi = -n\psi-mA_\phi\,,
\label{Psi_hellical}
\end{equation}
and the symmetry field
\begin{equation}
u = n\partial_\vartheta + m \partial_\phi \,.
\label{u_hellical}
\end{equation}
They are related by $i_u\beta = -d i_u A^\flat = d\Psi$. In~vector calculus, this is $B \times u = \nabla \Psi$, but~one has to use the metric to express it in the adapted toroidal coordinates.
However, for $n\ne 0$ (else the field is axisymmetric), $u$ is not volume-preserving:
\begin{equation}
L_u \Omega =  -2n \frac{r\sin\vartheta}{R_0-r\cos\vartheta} \, \Omega ,
\end{equation}
so we can not use Theorem \ref{thm3}.
Nonetheless, from~(\ref{eq:volfactor}), $u$ preserves the density $B^\phi$, so in addition to testing the formula of Theorem \ref{thm4}, we also test the new formula of Theorem \ref{thm3'}.

As in~\cite{KMM23,MM25}, the~particular field considered takes the following values and function for the vector potential~(\ref{eq:potential}).
\begin{align}
\label{eq:standard_values}
\begin{split}
\begin{aligned}
w_1& = 1/4,\\
w_2& = 1,
\end{aligned}\qquad
\begin{aligned}
B_0&=1,\\
R_0&=2,
\end{aligned}\qquad
\begin{aligned}
\zeta&= 0,\\
f(\psi)&= \psi-R_0^2/B_0.
\end{aligned}
\end{split}
\end{align}

\vspace{0.2cm}
The poloidal plane plots for this field will use  ``symplectic'' coordinates, defined as
\begin{align}
\label{eq:sym_cords}
\begin{split}
\tilde{y}& = \sqrt{2\psi/B_0}\cos\vartheta\\
\tilde{z}& = \sqrt{2\psi/B_0}\sin{\vartheta}
\end{split}
\end{align}
on the poloidal section $\phi=0$. Near~the magnetic axis, this is a small distortion (especially for a large aspect ratio $r/R_0 \longrightarrow 0$) of the true $yz$-plane $x=0$, but~with area equal to toroidal flux and the magnetic axis shifted to the~origin.



As can be observed from Figure~\ref{Fig_coords_comp_2_1}, the~label function $\Psi$ in~(\ref{Psi_hellical}) can switch between three different regions in the poloidal plane---inner tori, magnetic islands, and~outer tori---
which are to be 
distinguished. Therefore, the~volume computation should be carried out over different intervals of $\Psi$. Additionally, it is expected that the slow dynamics of fieldlines near the separatrices will yield divergent values for $T(\Psi)$; consequently, the~volume estimations given by~(\ref{V_general}) and {Theorems \ref{thm3'} and \ref{thm4}}  should be computed carefully, but~the integrals {are still expected to} converge.\vspace{-12pt}

\begin{figure}[ht!]
 \centering  
{\includegraphics[height=7.3cm]{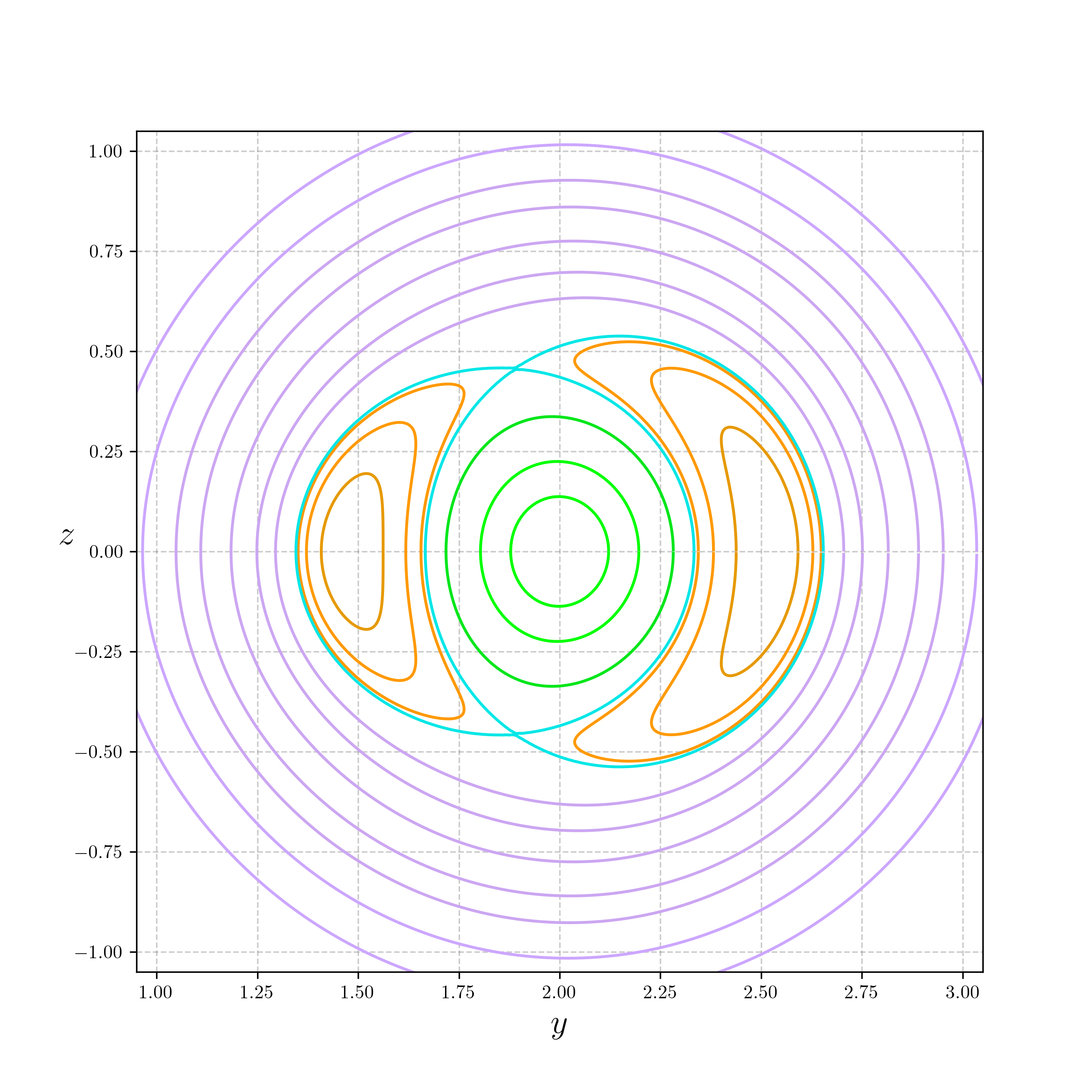}}
{\includegraphics[height=7.3cm]{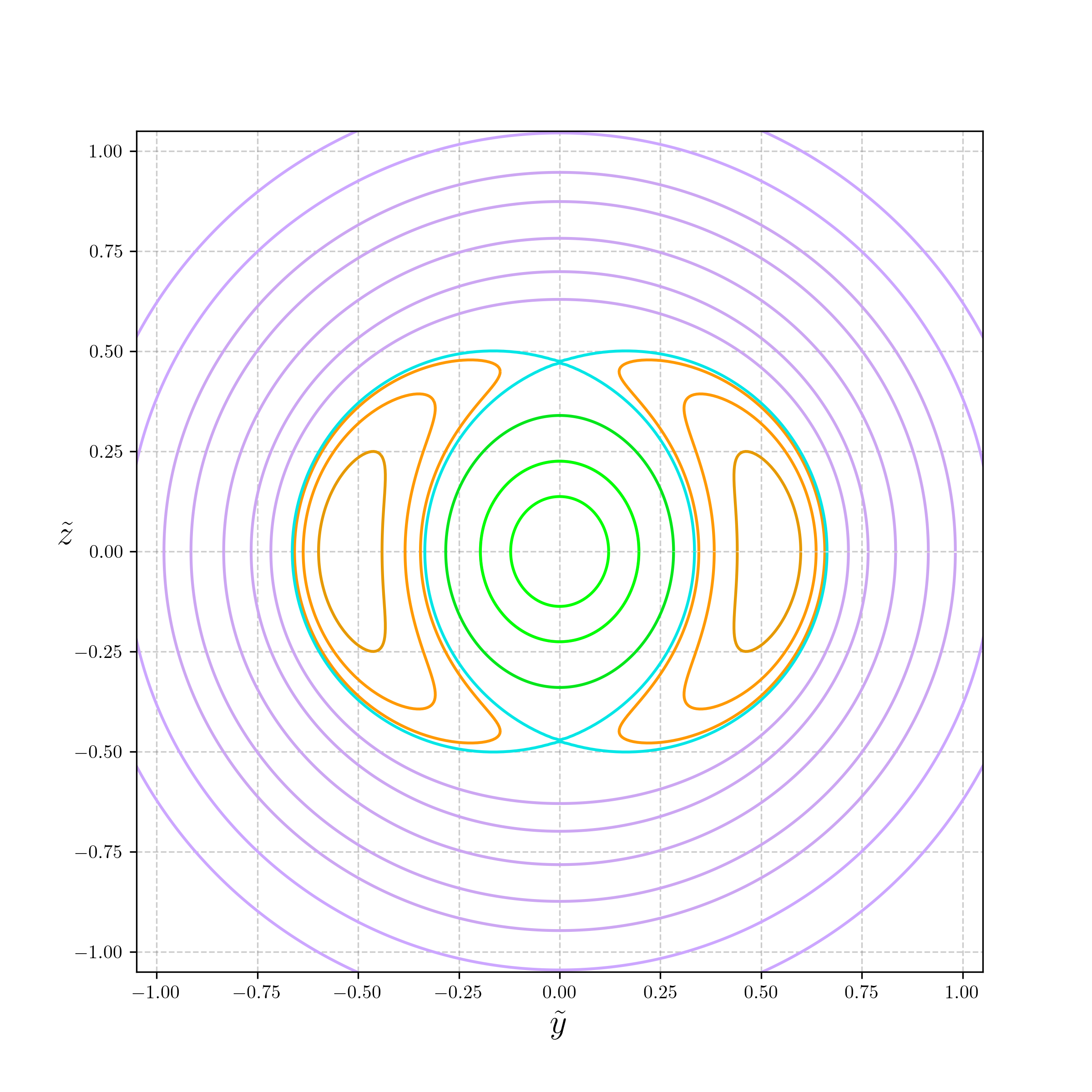}}
   \caption{{Level} 
 sets of $\Psi$ on the poloidal section $\phi=0$ for $(m,n)=(2,1)$ and standard values given by (\ref{eq:standard_values}), in~Cartesian (left) and symplectic (right) coordinates. The~tori are coloured according to the region: inner (green), magnetic island (yellow), outer (magenta) and the separatrices (cyan).}
   \label{Fig_coords_comp_2_1}
\end{figure}

\section{Results}
\label{sec:results}

In this section, we  apply the formulae in Equation~(\ref{V_general}) and Theorem~\ref{thm1} to the axisymmetric field (\ref{axisym_B}), and~the Equation~(\ref{V_general}) and Theorems \ref{thm3'} and \ref{thm4}  to the field (\ref{eq:fields}). As~previously mentioned, the~results in all  forthcoming figures are presented over the poloidal plane $\phi=0$. 

For the computation of the volume Formula (\ref{V_general}),
we employ a regular grid in a convenient set of Cartesian coordinates $(\mathsf{x},\mathsf{y})$ over the chosen poloidal plane, in~order to approximate the area element appearing in the volume formula.  
More precisely, for~the volume enclosed between the flux surfaces $\Psi_{0}$ and $\Psi_{1}$ (with $\Psi_{1} > \Psi_{0}$), we consider
\begin{equation}
\label{int_domain}
\mathcal{D}=\{(\mathsf{x}_m, \mathsf{y}_n) \in \mathbb{R}^2 : \Psi_0 < \Psi(\mathsf{x}_m,\mathsf{y}_n) < \Psi_1 \},
\end{equation}
where $(\mathsf{x}_m, \mathsf{y}_n)$ are the nodes of an $N_1 \times N_2$ regular grid over the domain  
$[\mathsf{x}_0 - L_1, \mathsf{x}_0 + L_1] \times [\mathsf{y}_0 - L_2, \mathsf{y}_0 + L_2].$
If $\beta$ in these coordinates (restricted to the poloidal plane) takes the form
\begin{equation}
\label{beta_P}
\beta_P = f(\mathsf{x},\mathsf{y}) \, d\mathsf{x} \wedge d\mathsf{y},
\end{equation}
then the volume Formula~(\ref{V_general}) is approximated by
\begin{equation}
    \label{eq:area}
    V \sim \frac{4L_1 L_2}{N_1 N_2}\sum_{i=1}^{N_1} \sum_{j=1}^{N_2} \mathsf{T}(\mathsf{x}_i,\mathsf{y}_j)\, 
    |f(\mathsf{x}_i,\mathsf{y}_j)|
    \,\cdot\, 
    1_{\{(\mathsf{x}_i,\mathsf{y}_j)\in \mathcal{D}\}}\, .
\end{equation}
In the first example, the~axisymmetric magnetic field, the~convenient coordinates are $(R,z)$ over the poloidal plane $\phi =0$. In~this case, $\beta$ restricted to the poloidal section takes the form:
\begin{equation}
    \label{beta_P_cil}
    \beta_P = -B^\phi\, R\,\, dR \wedge dz = -\frac{C}{R} \, dR \wedge dz \,,
\end{equation}
therefore, for~this case: $|f(R_i,z_j)| = C/R_i$.

For the second example, we employ the symplectic coordinates $(\tilde{y}, \tilde{z})$, since the areas are equivalent whether computed in $(\tilde{y}, \tilde{z})$ or in $(\psi, \vartheta)$,  by~
$d\tilde{y} \wedge d\tilde{z} = B_0^{-1} \, d\psi \wedge d\vartheta$.
Thus, $\beta$ restricted to the poloidal section $\phi=0$ can be written as,
\begin{equation}
\label{beta_P_hellical}
    \beta_P = B^\phi\,\sqrt{|g|}\, d\psi \wedge d\vartheta =  B_0\, d\tilde{y}\wedge d\tilde{z} \,;
\end{equation}
therefore, for~the second example: $|f(\tilde{y}_i,\tilde{z}_j)|=|B_0|$.

Note, however, that the accuracy of (\ref{eq:area}) cannot be expected to be better than $O((N_1 N_2)^{-3/4})$; see for example, the~classic cases of Gauss and Dirichlet for the number of grid points inside a circle or under a hyperbola~\cite{Hux}. 
Partly for this reason, and~partly to see if we could improve the efficiency of computation of $V(\Psi)$ as a function of $\Psi$ (not just individual values of $\Psi$), we also tested the method in the Appendix. 
For this purpose, we established an algorithm to determine the level sets of $\Psi(\tilde{y},\tilde{z},\phi=0)=\Psi_0$ on the poloidal plane, 
which then is discretised  into a regular partition of $N_g$ points $\lbrace \nu_j \rbrace$. With~it, we computed (\ref{eq:vol_lambda}) by,
\begin{equation}
\label{eq:V_lambda_disc}
    \frac{dV}{d\Psi}(\Psi_0) \sim \sum_{j=1}^{N_g} 
    \mathsf{T}(\nu_j) i_{\epsilon_j} \lambda \,,
\end{equation}
where $\epsilon_j$ denotes half displacement between $\nu_{j-1}$ and $\nu_{j+1}$, and~$\lambda = i_B i_n \Omega$, with~$n= \nabla \Psi/|\nabla\Psi|^2$. To~simplify the computations for the second example, we used the diagonal metric
$$ds^2 = \frac{1}{2B_0\psi} d\psi^2 + \frac{2\psi}{B_0} d\vartheta^2 + R_0^2 d\phi^2,$$
to compute $\nabla\Psi$ and~$i_{\epsilon_j} \lambda = \det W /B^\phi$, where $W$ is a matrix whose columns are the contravariant components of $(n, B, \eps_j)$ in the $(\psi, \vartheta,\phi)$ coordinate~system.

The implementation of (\ref{eq:vol_lambda}) for the first example used a discretisation $\nu_j$ of the level sets of $\Psi$ with coordinates $R_j = (1+\sqrt{2\Psi}\cos\theta_j), z_j = \sqrt{2\Psi}\sin\theta_j$
and $\theta_j = 2\pi/N$, $j=0,1,\dots,N-1$. Following the steps in the Appendix A, 
(\ref{eq:V_lambda_disc}) reduces in this case to
\begin{equation}
    \frac{dV}{d\Psi}(\Psi_0) 
    \sim C \sum_{j=0}^{N-1} \frac{\mathsf{T}(\nu_j) \sin(2\pi/N)}{1+\sqrt{2\Psi}\cos\theta_j}\,.
\end{equation}

In contrast, for~the integration of the volume formulae in Theorems \ref{thm1}, \ref{thm3'} and \ref{thm4}, and~the integration of (\ref{eq:V_lambda_disc}), the~trapezoidal rule was applied over a uniform grid in $\Psi$, and~the LSODA algorithm (Python's {\tt scipy} implementation) was employed for the field-line integration. All codes used an a priori short integration time for the fieldline (to get away from the initial line) and then extended it until the minimum number of crossings, $N_c$, required by the corresponding formula was~attained.

Details about the codes used in the computations are provided in the Supplementary Materials.

\subsection{Axisymmetric Magnetic Field---Theorem~\ref{thm1}}
\label{subsec:res_theo1}

The left subplot in Figure~\ref{Fig_V_axissym} shows a comparison between the computations of the volume enclosed by the flux surfaces using the general Formula~(\ref{V_general}) and Theorem~\ref{thm1}, for~the axisymmetric magnetic field~(\ref{V_general}), as~a function of the flux function $\Psi$, defined in~(\ref{flux_func_axisymm}). The~analytic volume given in~(\ref{vol_axisym}) is also included. The~volume formula~(\ref{V_general}) was computed over a regular array of points contained within a disc $D$ of radius $\sqrt{2\Psi}$, with~$\Psi$ taken from~(\ref{flux_func_axisymm}). For~the volume computation using Theorem~\ref{thm1}, only one initial point $x = (1 + \sqrt{2\Psi}, 0)$ was used to compute $T(\Psi)$. The~right inset in Figure~\ref{Fig_V_axissym} shows an example of the regular grid used to compute (\ref{vol_axisym}) and the flux surfaces corresponding to a uniform discretisation on $\Psi$ used in Theorem \ref{thm1}.

\begin{figure}[ht!]
 \centering  
{\includegraphics[height=8.0cm, trim={20 0 10 0},clip]{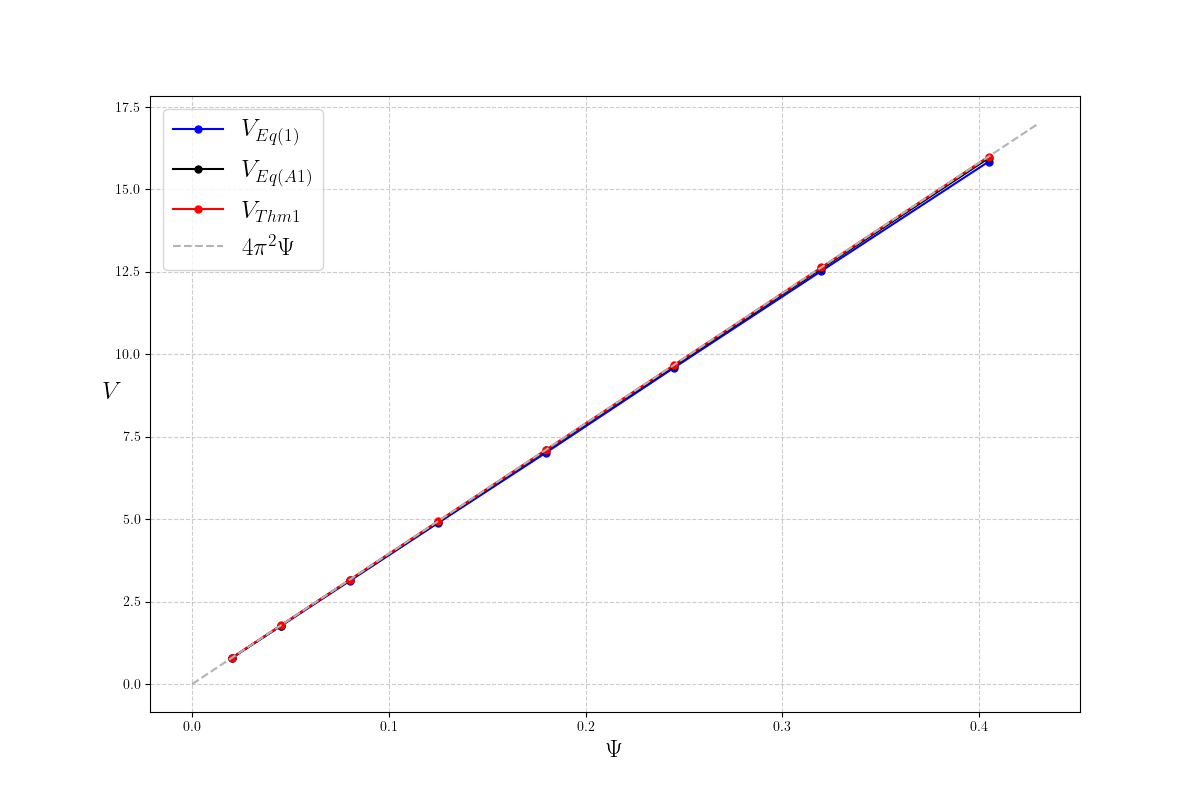}}
{\includegraphics[height=8.0cm, trim={0 0 0 0},clip]{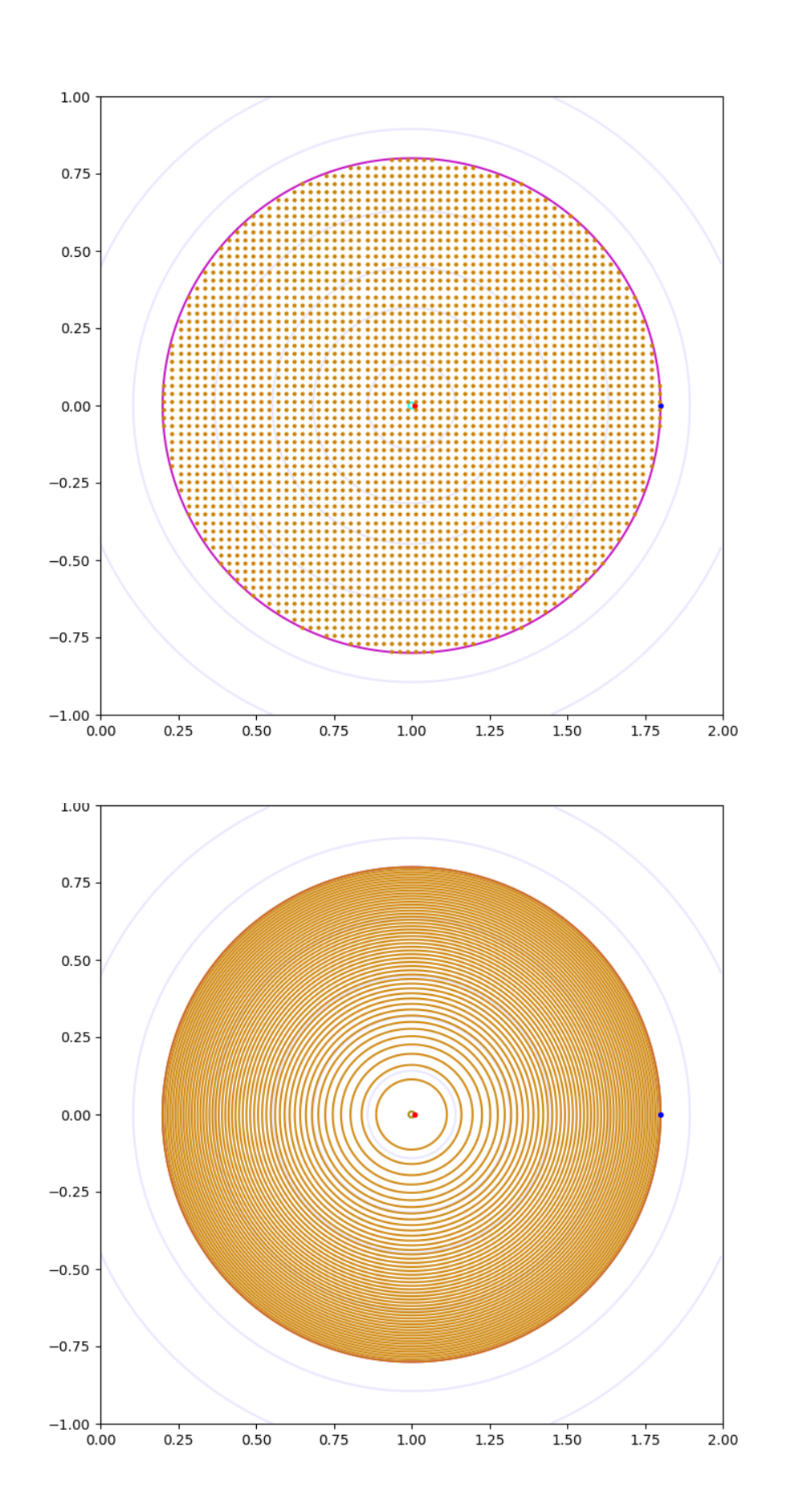}}
   \caption{(Left) Volume enclosed by the flux-surface computed by (\ref{V_general}) (blue), (\ref{eq:vol_lambda})  (black), Theorem~\ref{thm1} (red) and exact formula (\ref{vol_axisym}) (dashed grey) as a function of $\Psi$ for the field (\ref{axisym_B}) , with~$C=1$. (Right) {Example} 
 of the grid ({top}) and set of flux surfaces ({bottom}) used to compute (\ref{V_general}) and Theorems \ref{thm1}, respectively.}
   \label{Fig_V_axissym}

\end{figure}
\unskip

\subsection{Toroidal Helical Magnetic Field---Theorems \ref{thm3'} and~\ref{thm4}}
\label{subsec:res_theo3}

The example considered corresponds to the magnetic field derived from (\ref{eq:potential}) for one mode, namely the resonance $2/1$, that is
\begin{equation}
\label{A_ex2}
A_\phi = -\left[\psi/4 + \psi^2 + \eps\psi(\psi-4)\cos(2\vartheta-\phi) \right].
\end{equation}
Thus, (\ref{Psi_hellical}) becomes $\Psi = -\psi - 2 A_\phi$, which implies that vector field $u$ takes the form,
\begin{equation}
\label{u_hellical_21}
u = \partial_\vartheta + 2 \partial_\phi \,.
\end{equation}  

Write
\begin{equation}
    v = B/\rho \,,
    \label{eq:v}
\end{equation}
with $\rho = B^\phi$.

As the vector field $u$ in (\ref{u_hellical_21}) is constant with integer components (and thus commensurate), the~$u$-lines are closed and all have period $2\pi$ in this case. Consequently, one of the lattice generators can be written as ${T}_1 = (2\pi, 0)^{\mathrm{T}}$ (superscript $^T$ denotes transpose, and~flowing for time $2\pi$ along $u$ and for $0$ along $v$ produces the identity map on any flux surface).
The generator ${T}_2$ can be obtained through the return time ${T}$ along $v$ to any chosen $u$-line, using ${T}_2 = (-c, {T})^{\mathrm{T}}$, where $-c$ is the time needed to flow along $u$ to reach the initial point again. 
This is illustrated in Figure~\ref{T12_diagram} for the tori in the inner and outer regions.  Thus, the~determinant in the volume formula for this case (Theorem \ref{thm3'}) takes the form $$\Delta(\Psi) = 2\pi T(\Psi)\,,$$ 
where $T(\Psi)$ is the return time to the $u$-line along the scaled field $v$.
The value of $c$ is irrelevant for this.   Thus, one factor in the volume formula for this example happens to take a similar form as in Theorem \ref{thm1}. 
\begin{figure}[ht!]
 \centering  
{\includegraphics[height=5.2cm, trim={0 -20 0 0},clip]{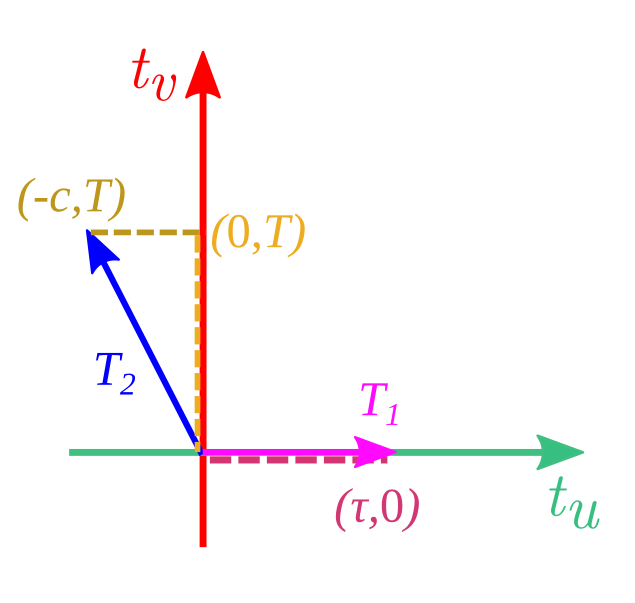}}
{\includegraphics[height=6.2cm, trim={0 10 50 0},clip]{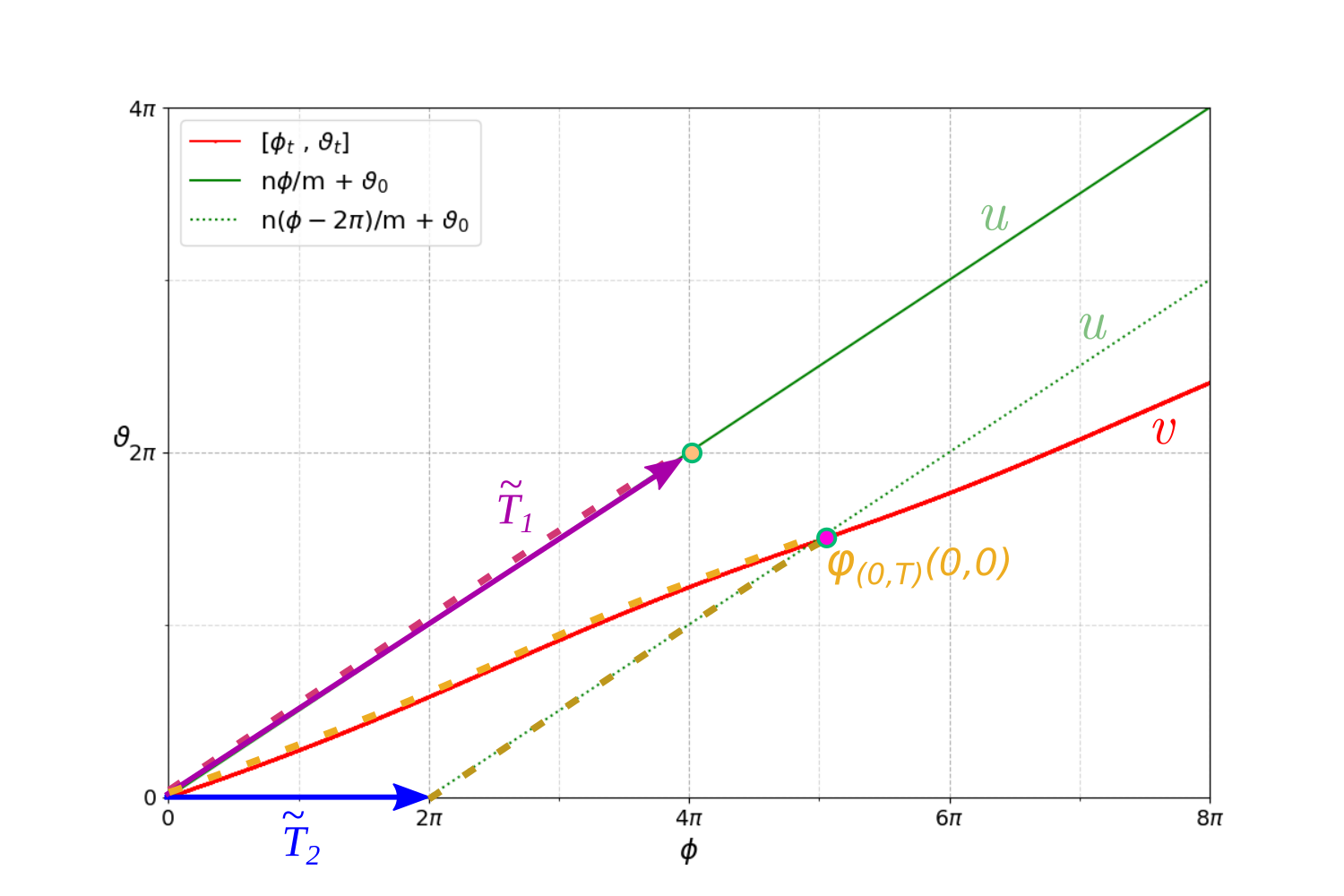}}
\caption{
   {Selection}
 of the lattice generators $T_{1}$ and $T_{2}$ for a torus in the inner region. (Left): Generators {$T_j$} in $(t_{u}, t_{v})$-coordinates, corresponding to the flow times under the $u$- and $v$-fields. (Right): Vectors $\tilde{T}_{j}$, the~images of $T_{j}$ under $\varphi_T(0,0)$ in the $(\phi, \vartheta)$-plane.  $\tilde{T}_1$ (in purple) is drawn over an orbit of $u$ (in green) 
   through the initial point $(0,0)$; one translate of the $u$-orbit is also shown in green. The~construction of $\tilde{T}_2$ requires the flow of $T$  units under $v$ (along the red curve)  and then $-c$ units under $u$. 
   }
   \label{T12_diagram}
\end{figure}

The other factor, namely  $1/\hat{\rho}$, is computed by averaging $1/\rho$ over a regular subdivision of the lattice $\Gamma$, as~described in Section~\ref{subsec:int_fields_dens}.   
Figure~\ref{fig:avrho} shows  $1/\hat{\rho}$ computed in the three regions for subdivisions by $q=4-12$.
We see convergence for practical purposes when $q=6$, except~near the separatrix.  {The}~convergence near the separatrix is less good, because~analyticity of the AL coordinates fails in the separatrix limit, but~the true value of $1/\hat{\rho}$  is just the time-average of $1/\rho$ along the hyperbolic closed fieldline ($\approx$1.94399).   
The approach of $1/\hat{\rho}$ to the separatrix value is consistent with the form $\delta \Psi \log \delta \Psi$, as~is to be expected from the typical logarithmic divergence of return times near a hyperbolic periodic orbit. 
We use $q=6$ in our calculations of enclosed~volume.

In
 the case of tori within a magnetic island, the~visualisation on the $(\phi,\vartheta)$ plane is insufficient, as~these coordinates are not well suited to describing an island torus. In~these coordinates, the~$v$-orbit appears to oscillate around the $u$-line, as~shown in the right-hand inset of Figure~\ref{T12_diagram_island}. 
But the first apparent intersection of the $v$-orbit with the $u$-line is actually an illusion caused by a two-to-one projection of the flux surface to $(\phi,\vartheta)$. This apparent first crossing triggers the code to search for a second crossing ($N_c = 2$, instead of $1$ as in the two other regions).

For Theorem~\ref{thm4}, choosing $\gamma$ to be a $u$-line (i.e., $\psi$ = constant and $2\vartheta-\phi$ = constant), 
we have that,
\begin{align}
\label{Phi_computation}
\begin{split}
    \Phi =& \int_\gamma A^\flat = 
    \int_\gamma (A_\psi d\psi + A_\vartheta d\vartheta 
               + A_\phi d\phi )  \\
    =& \int_\gamma \{\psi d\vartheta 
        -\left[\psi/4 + \psi^2 + \eps\psi(\psi-4)\cos(2\vartheta-\phi) \right] d\phi )\}\\
    =& 2\pi \psi - 4\pi[\psi/4 + \psi^2 
         + \eps\psi(\psi-4)\cos(2\vartheta_0-\phi_0)] \\
    =& -2\pi \Psi(\psi,\vartheta_0,\phi_0) \,.
\end{split}
\end{align}
Therefore, the~volume formula for this example again takes a form similar to that in Theorem~\ref{thm1}: $dV = 2\pi \bar{T}\, d\Psi,$
where the average $\bar{T}$ is computed using the unscaled magnetic field $B$. We found it to be sufficiently accurate when estimated by averaging the first ten crossings.

\vspace{-12pt}
\begin{figure}[ht!]
 \centering  
{\includegraphics[height=5cm, trim={10 0 30 0},clip]{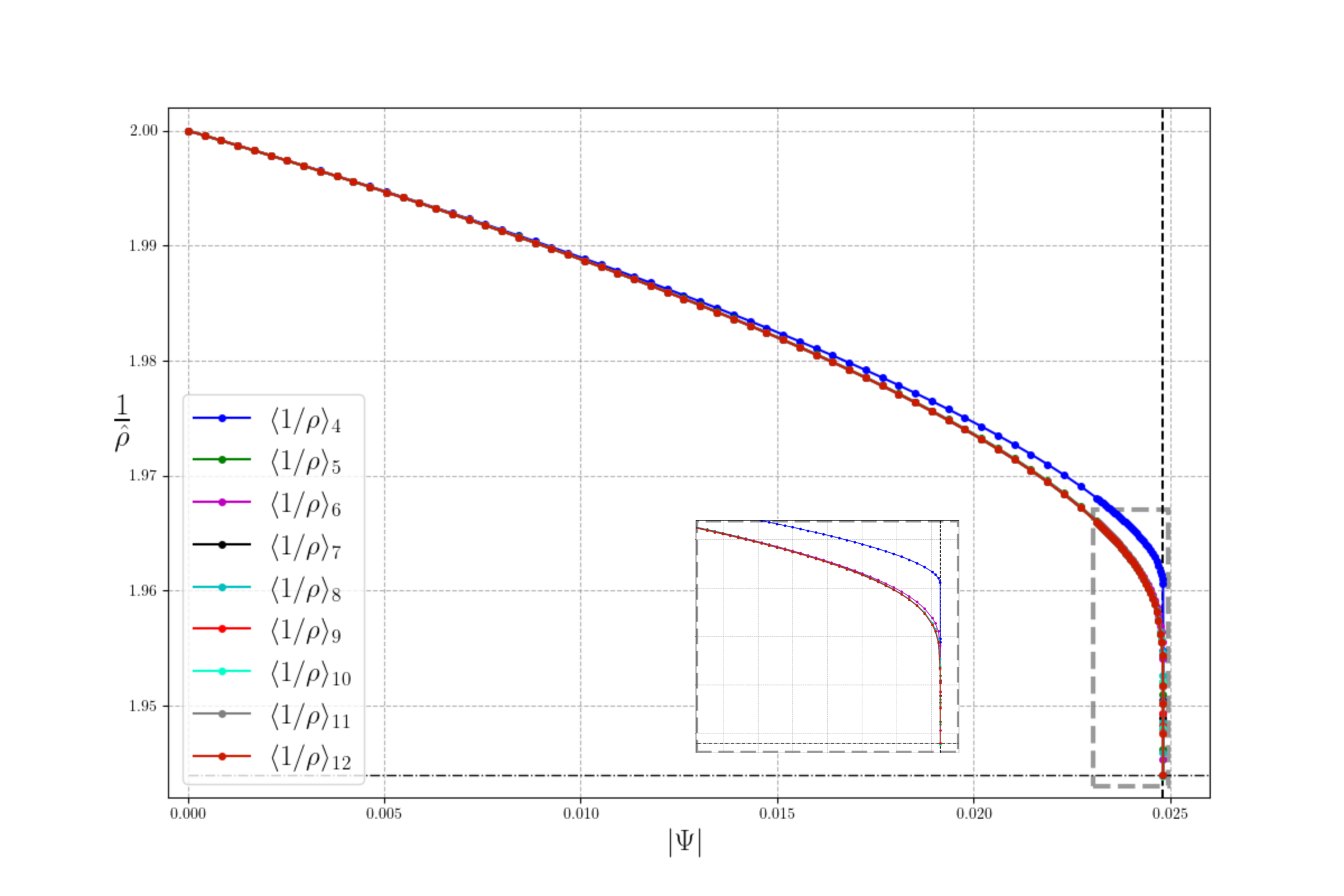}}
{\includegraphics[height=5cm, trim={10 0 30 0},clip]{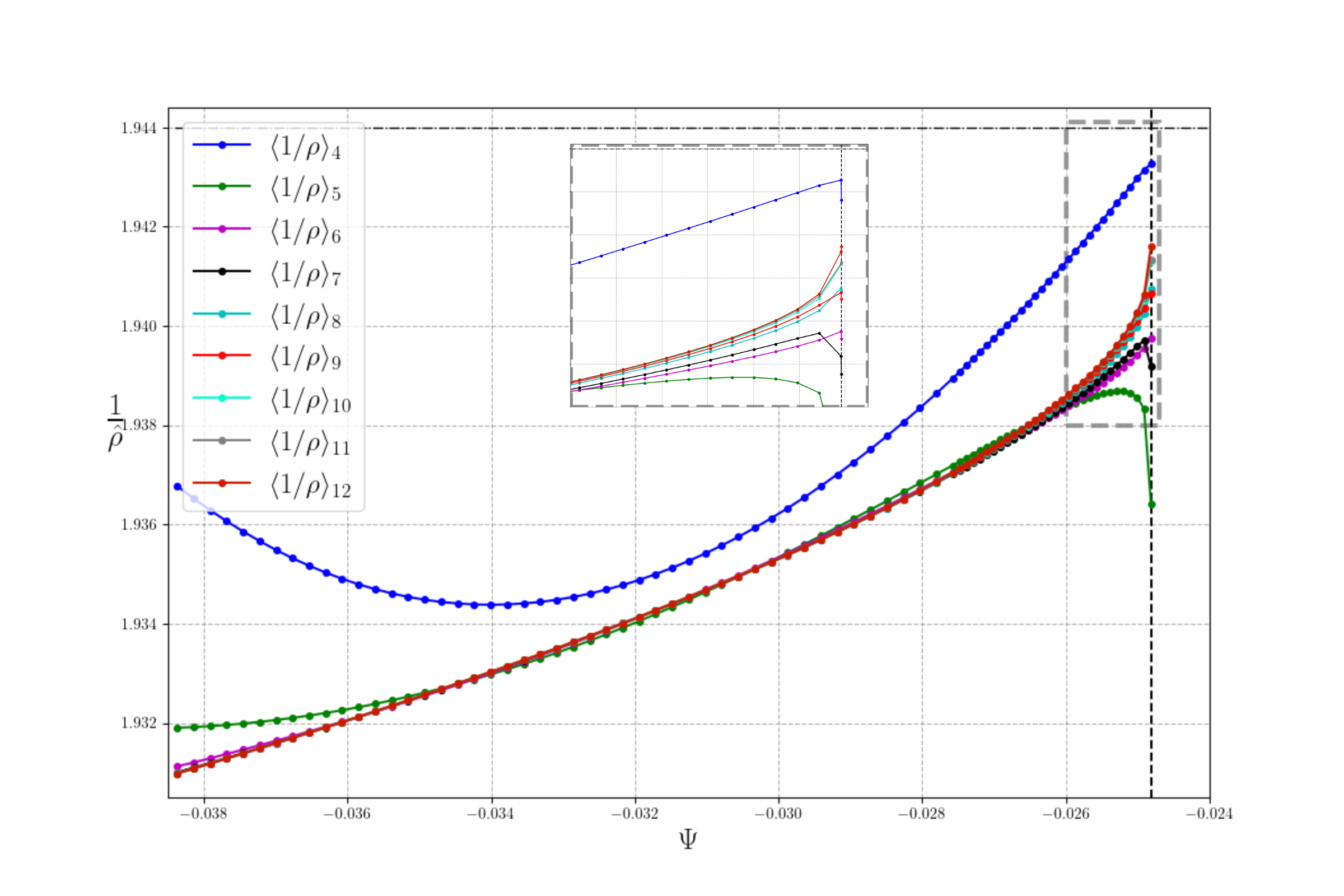}}\\
{\includegraphics[height=5cm, trim={10 0 30 0},clip]{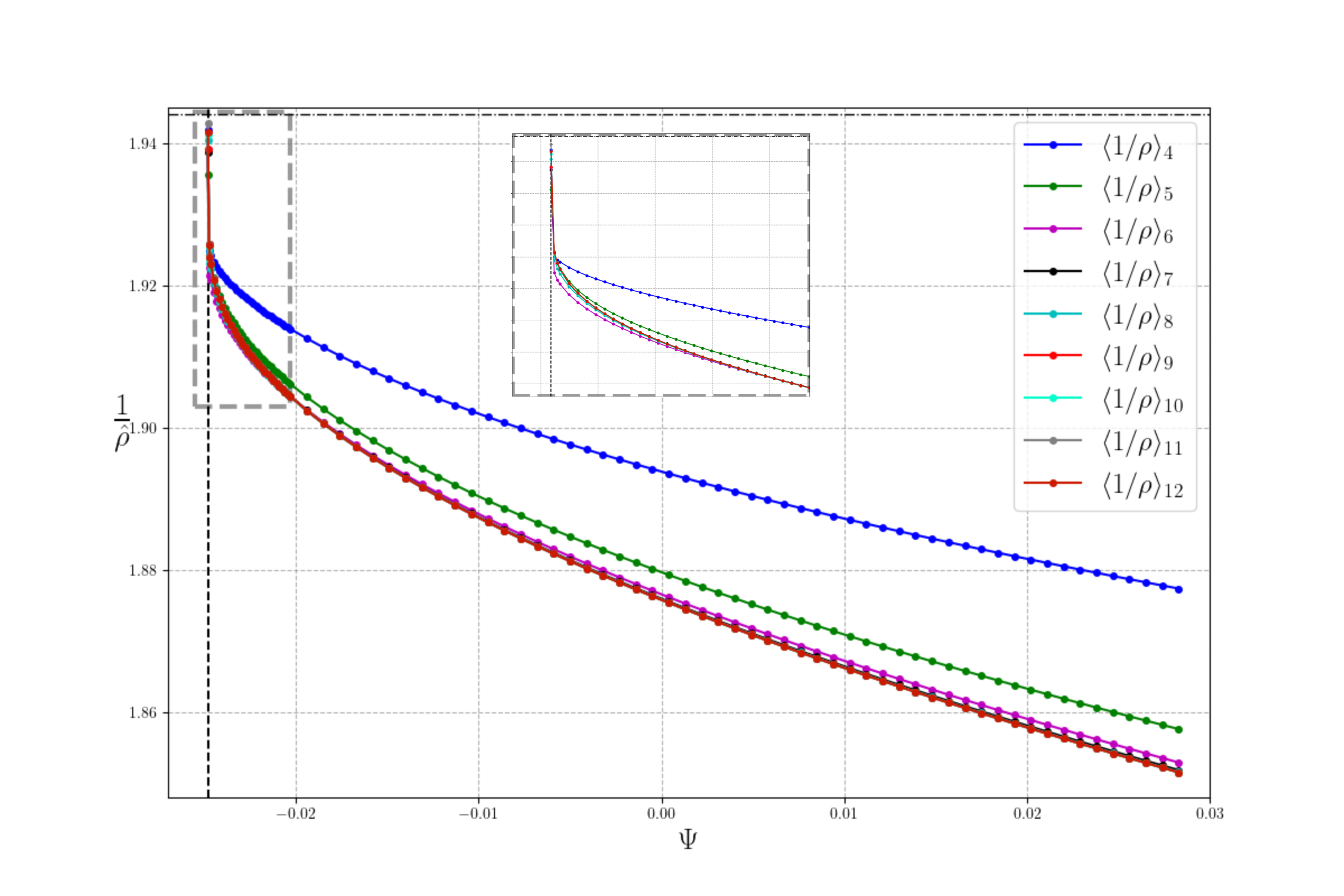}}
\caption{{Average} 
 $1/\hat{\rho}$ computed for subdivisions of sizes $q=4 - 12$: $\langle 1/\rho \rangle_q$, as~a function of $\Psi$ in the three regions:   (upper-left) inner, (upper-right) magnetic island,  and (bottom)
outer  regions.  The~vertical lines indicate the value of $\Psi$ for the separatrix, and~the horizontal dash--dot line indicates the numerical value of $\langle 1/\rho\rangle$ on the hyperbolic closed fieldline. Details of the plots near the separatrix are shown in dashed insets.
}
   \label{fig:avrho}

\end{figure}

\vspace{-12pt}
\begin{figure}[ht!]
\centering  
 \includegraphics[height=7.0cm, trim={20 0 50 0},clip]{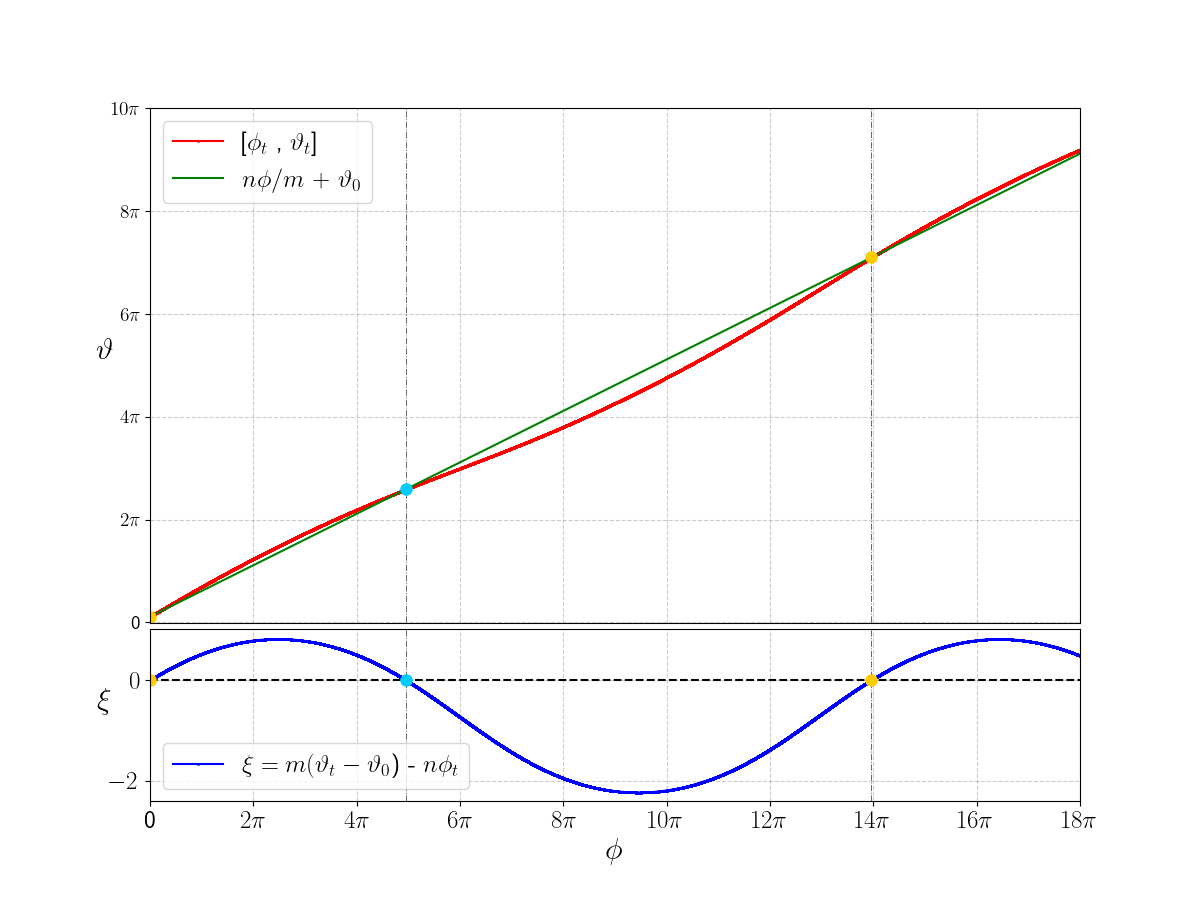} 
\caption{Plot
 in the $(\phi,\vartheta)$ plane of the orbit (red) and the $u$-line (green) starting at the point $(\tilde{y},\tilde{z}) = (0.570, 0.211)$ within the magnetic island. The~first intersection in this plot, between~the orbit and the $u$-line, 
    does not correspond to an actual crossing, but~rather to the orbit passing through a different point $(\psi_1,\vartheta_1,\phi_1)$ on the flux surface that satisfies $\vartheta_0 = \vartheta_1-n\phi_1/m$ but not $\psi_1= \psi_0$.}
   \label{T12_diagram_island}
\end{figure}

%
Figure~\ref{avT_regions} shows the averages of the return time, $\langle T\rangle_N$, over~the first $N$ crossings as a function of $\Psi$, with~$N=1, 10, 20, 30$, in~the three regions: inner, island, and~outer. The~first plot uses $|\Psi|$ on the $x$-axis (rather than $\Psi$, which is negative there) to 
make it go from the magnetic axis on the left to the separatrix on the right. The~value of $\Psi$ corresponding to the separatrix is indicated by a vertical dashed line in all three subplots, coinciding with the divergence of $T$. 

\begin{figure}[ht!]
 \centering  
{\includegraphics[height=5.5cm, trim={10 0 30 0},clip]{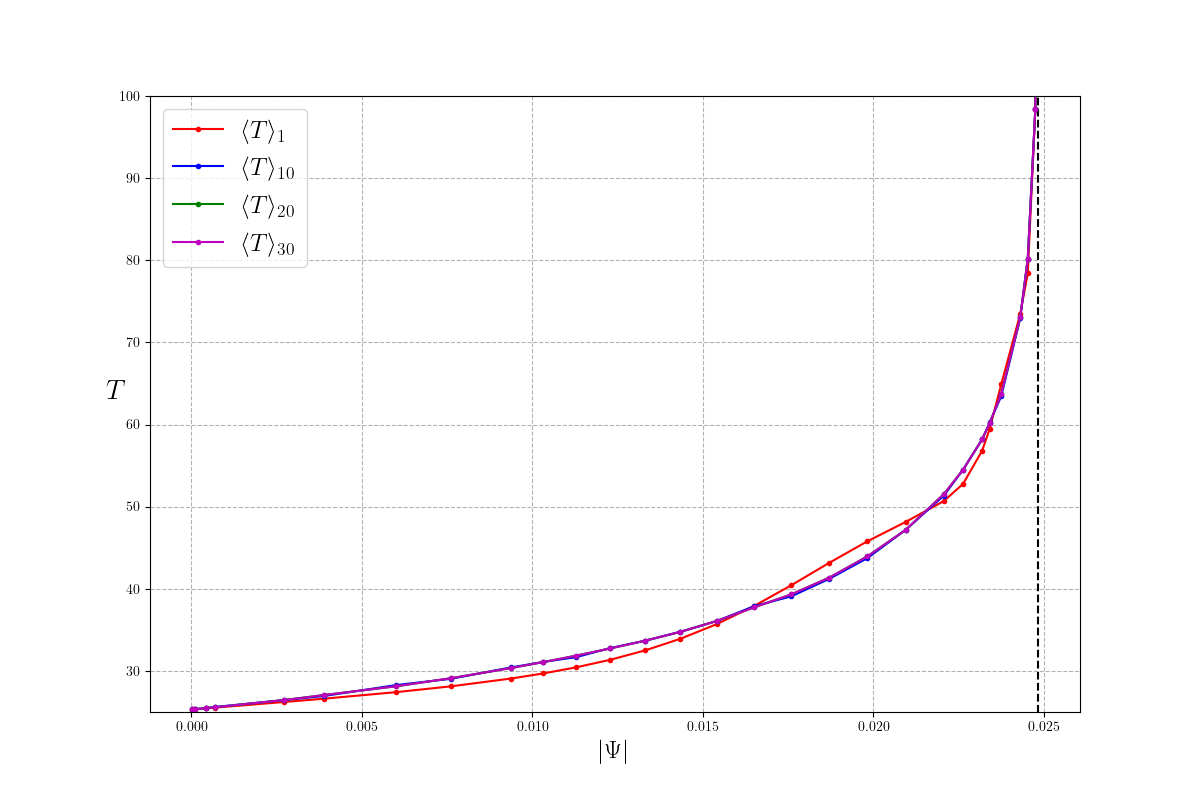}}
{\includegraphics[height=5.5cm, trim={10 0 30 0},clip]{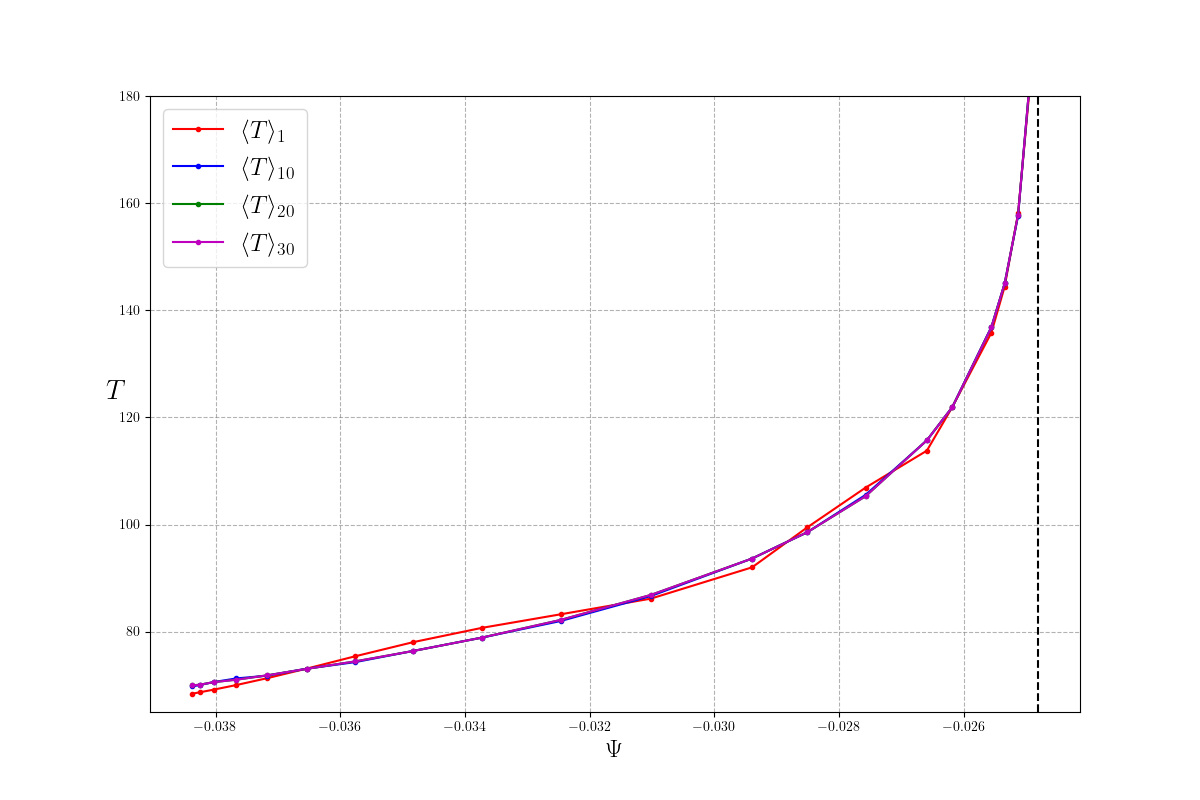}}\\
{\includegraphics[height=5.5cm, trim={10 0 30 0},clip]{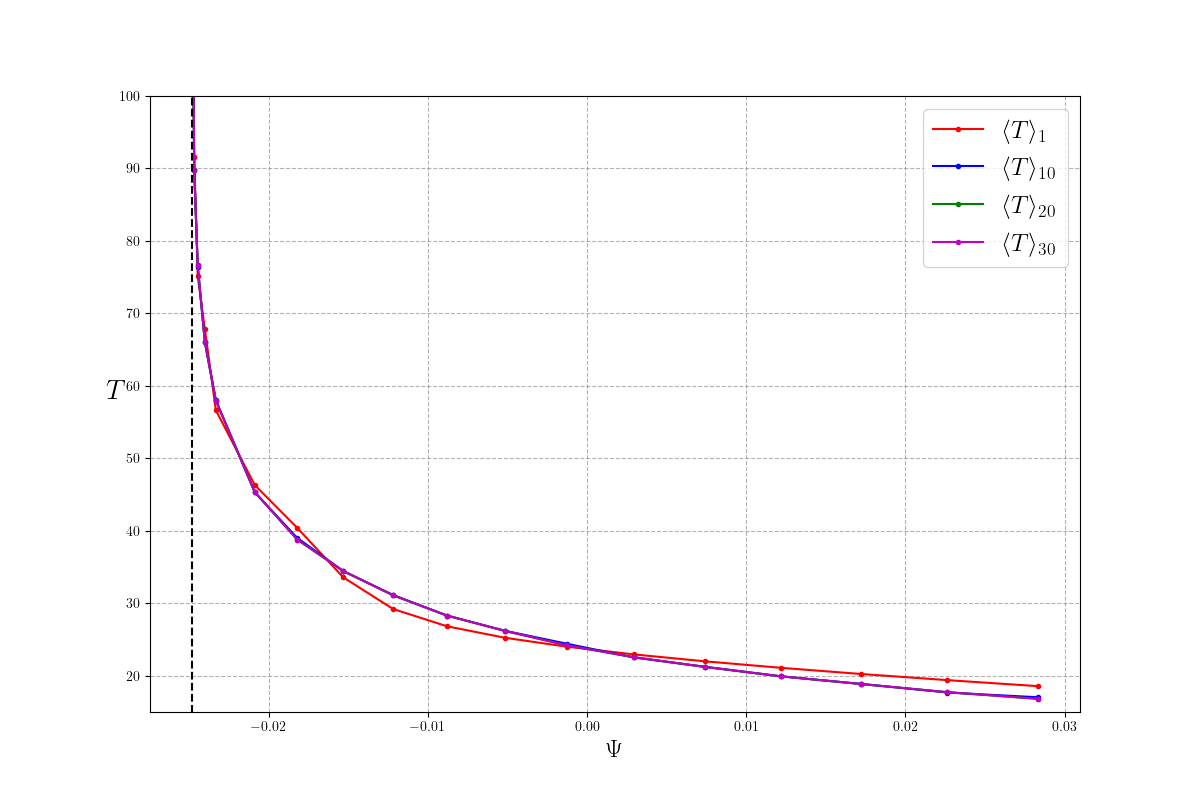}}
\caption{{Averages} 
of  the return time $\langle T\rangle_N$ of an unscaled field-line to a $u$-line as a function of $\Psi$ in the three regions:  inner (upper-left), magnetic island (upper-right) and outer (bottom) regions. The~vertical lines indicate the value of $\Psi$ for the separatrix.}
\label{avT_regions}
\end{figure}

In the next subsections, we give a little more detail about how we determine the lattice generators and give the results for the enclosed volumes for~the three types of flux surface in this~example.
\FloatBarrier                   

\subsubsection{Inner~Region}
For the inner region, a~typical computation of $T_2(\Psi)$ is displayed in Figure~\ref{Fig_T2_inner}.
From the theory, the~return time along $v$ to a given $u$-line is independent of the initial point on the $u$-line, but~we checked this numerically and confirmed it. Also, the~return time $T_2$  diverges like $- \log \delta \Psi$ as a separatrix is approached (compare the period of a pendulum), Figure~\ref{Fig_T2_log}.

\begin{figure}[ht!]
 \centering  
{\includegraphics[height=6cm, trim={60 0 20 0},clip]{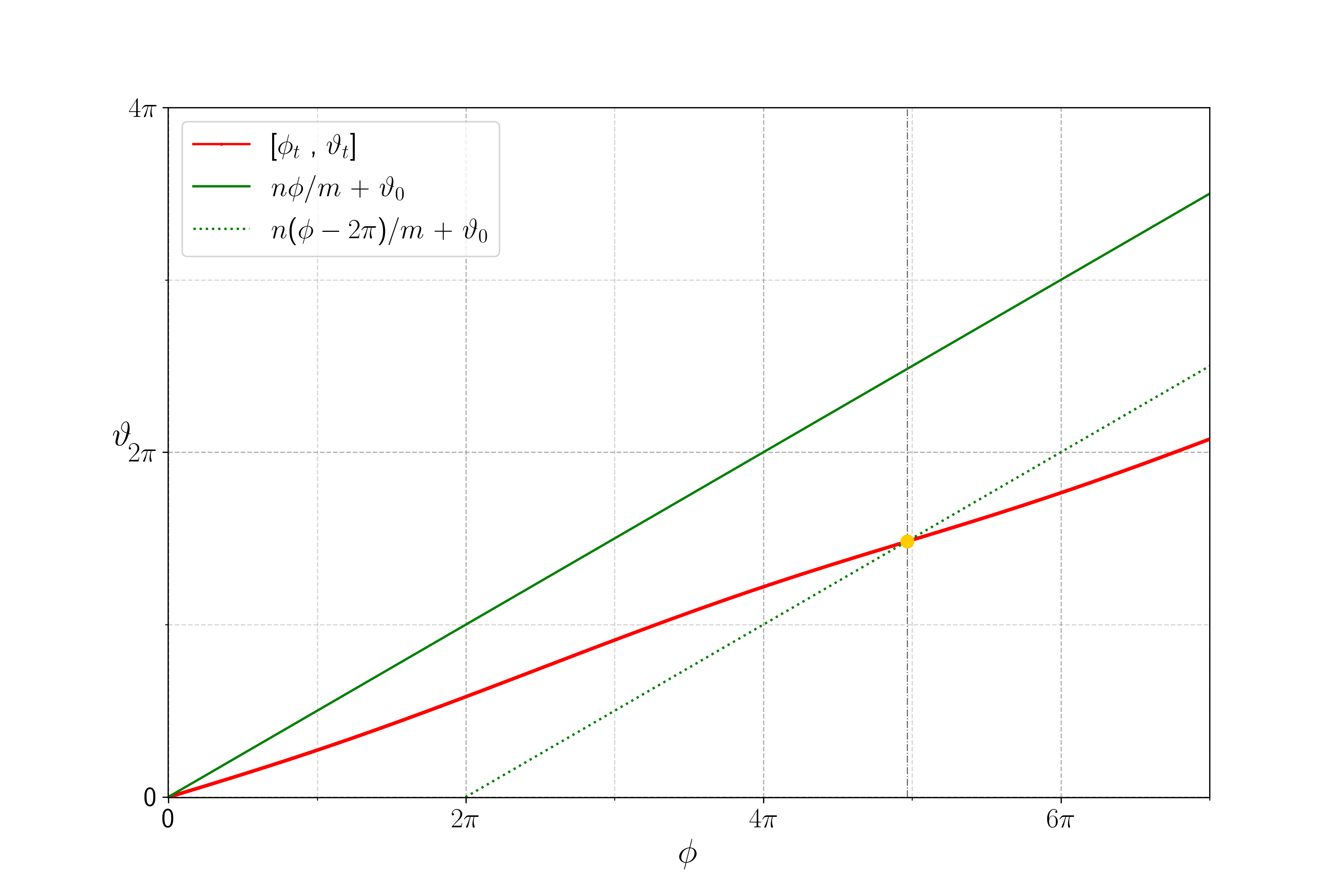}}
{\includegraphics[height=6cm, trim={10 0 70 0},clip]{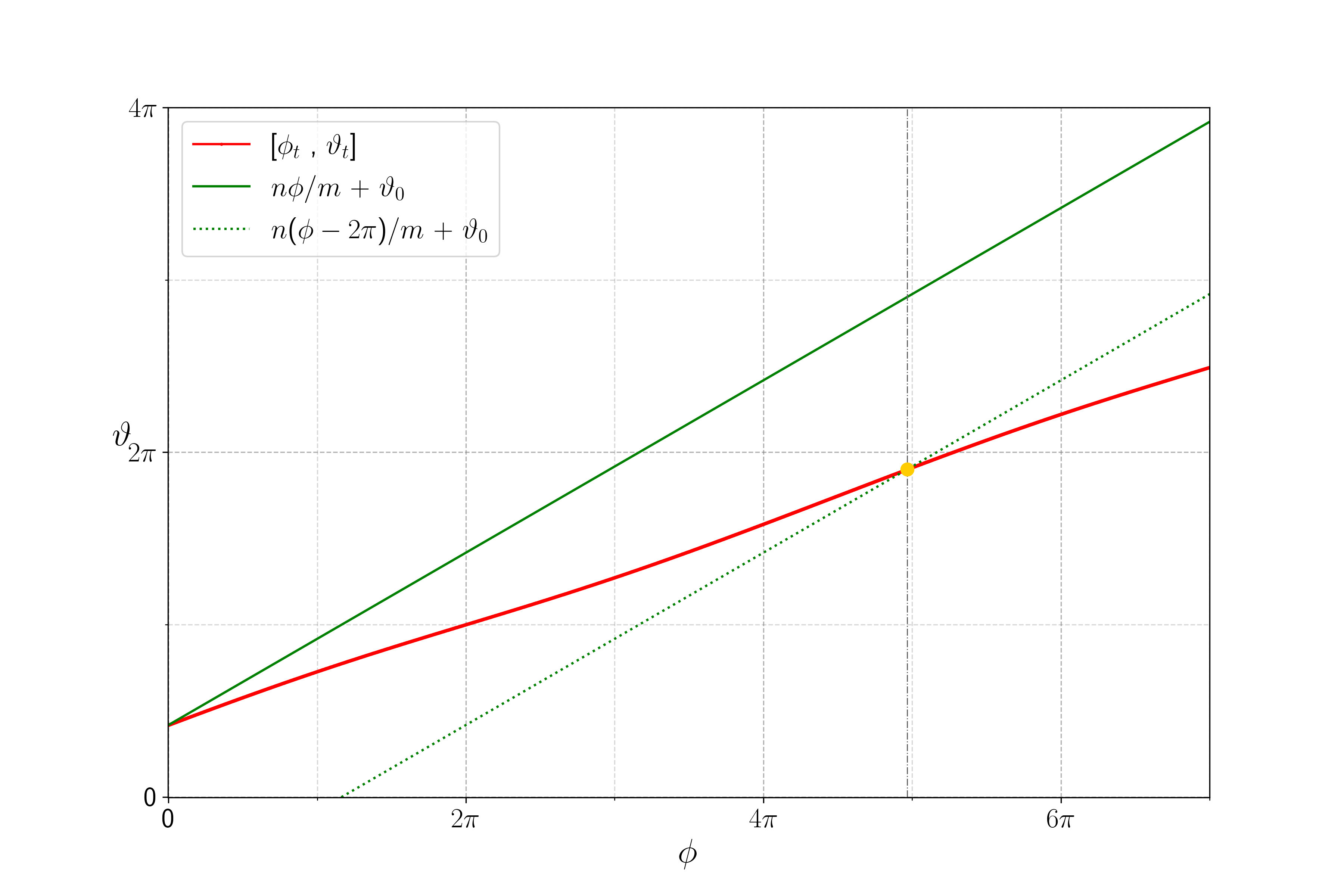}}  
\caption{
  Crossing 
 of the fieldline (red) and images of the $u$-line (green), starting at two selected initial points $(\tilde{y}_i, \tilde{z}_i)$ on the poloidal plane $\phi = 0$: (left) $(0.2,0)$ and (right) $(0.06,0.21)$ on the same flux surface ($\Psi=-0.01031$) in the inner region, for~$\varepsilon=0.007$. The~return time along $v$ is the same in both subplots $T(\Psi)=15.60$.}
   \label{Fig_T2_inner}
\end{figure}
\unskip

\begin{figure}[ht!]
\centering  
\includegraphics[height=8cm, trim={0 -20 0 0},clip]{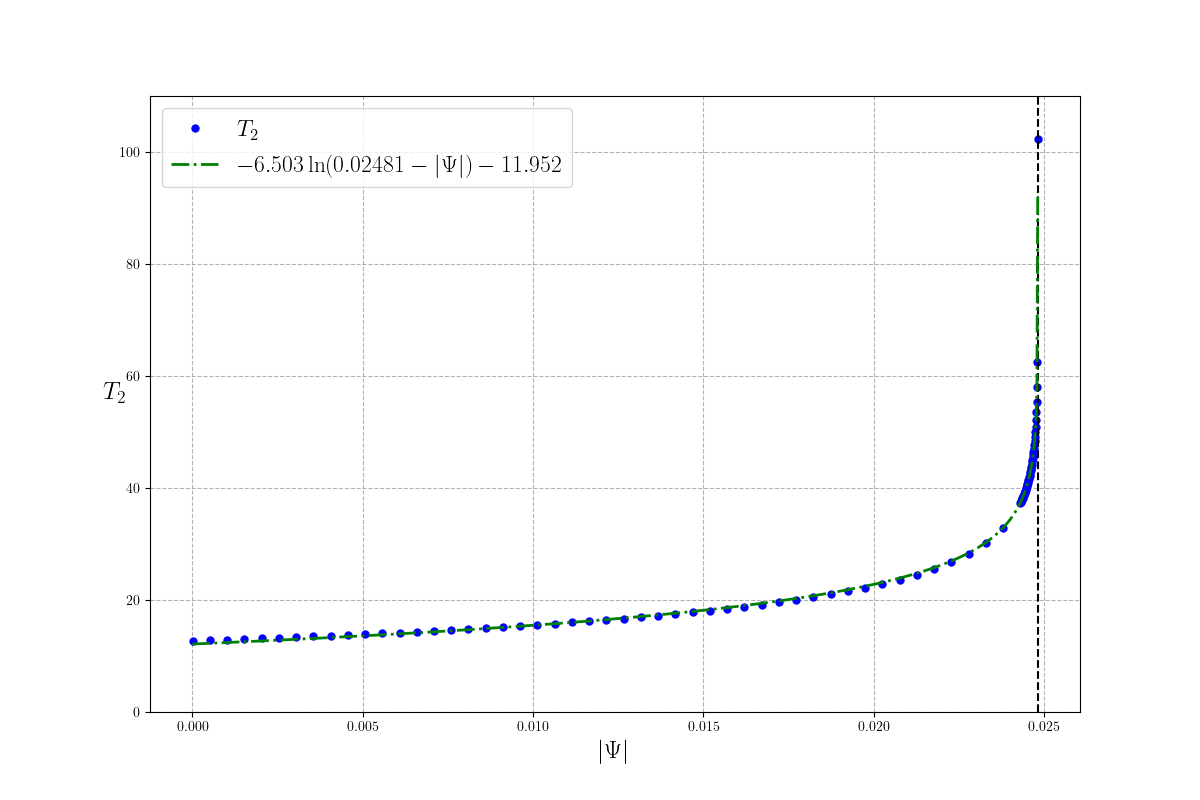}
\caption{
{Return} 
 time $T_2$ along $v$, computed in the inner region  (blue), and~the  fitted logarithmic curve $f(\Psi) = -A \ln(S-|\Psi|) + B$ (green dash-dot line), with~$S = 0.0248157$ corresponding to the separatrix.}
   \label{Fig_T2_log}
\end{figure}

Figure~\ref{Fig_V_inner} shows the volume computed by (\ref{V_general}) and by Theorem \ref{thm3} for~the inner region as a function of $\Psi$ defined in (\ref{Psi_hellical}).
We expected to see a more evident $\delta\Psi \log \delta\Psi$ behaviour near the separatrix but it requires zooming in (not shown).

\begin{figure}[ht!]
 \centering  
{\includegraphics[height=8.0cm, trim={20 0 10 0},clip]{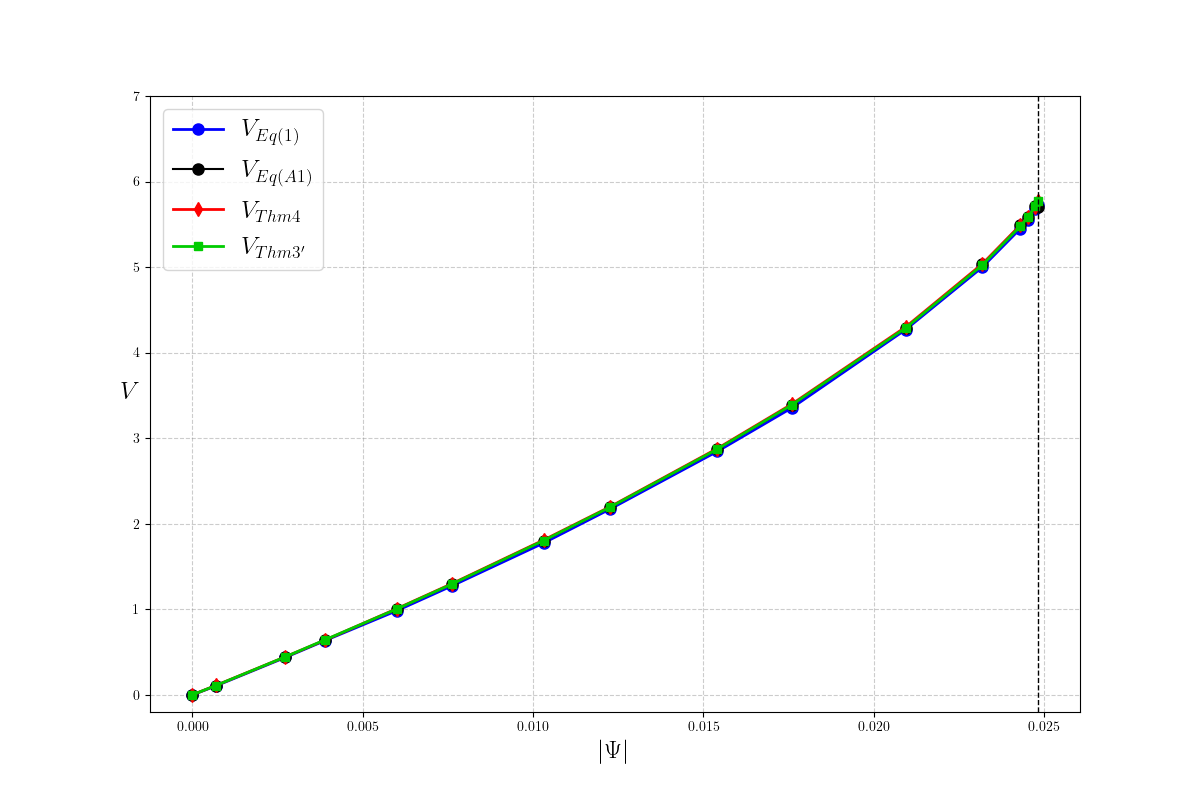}}
{\includegraphics[height=8.0cm, trim={0 0 0 0},clip]{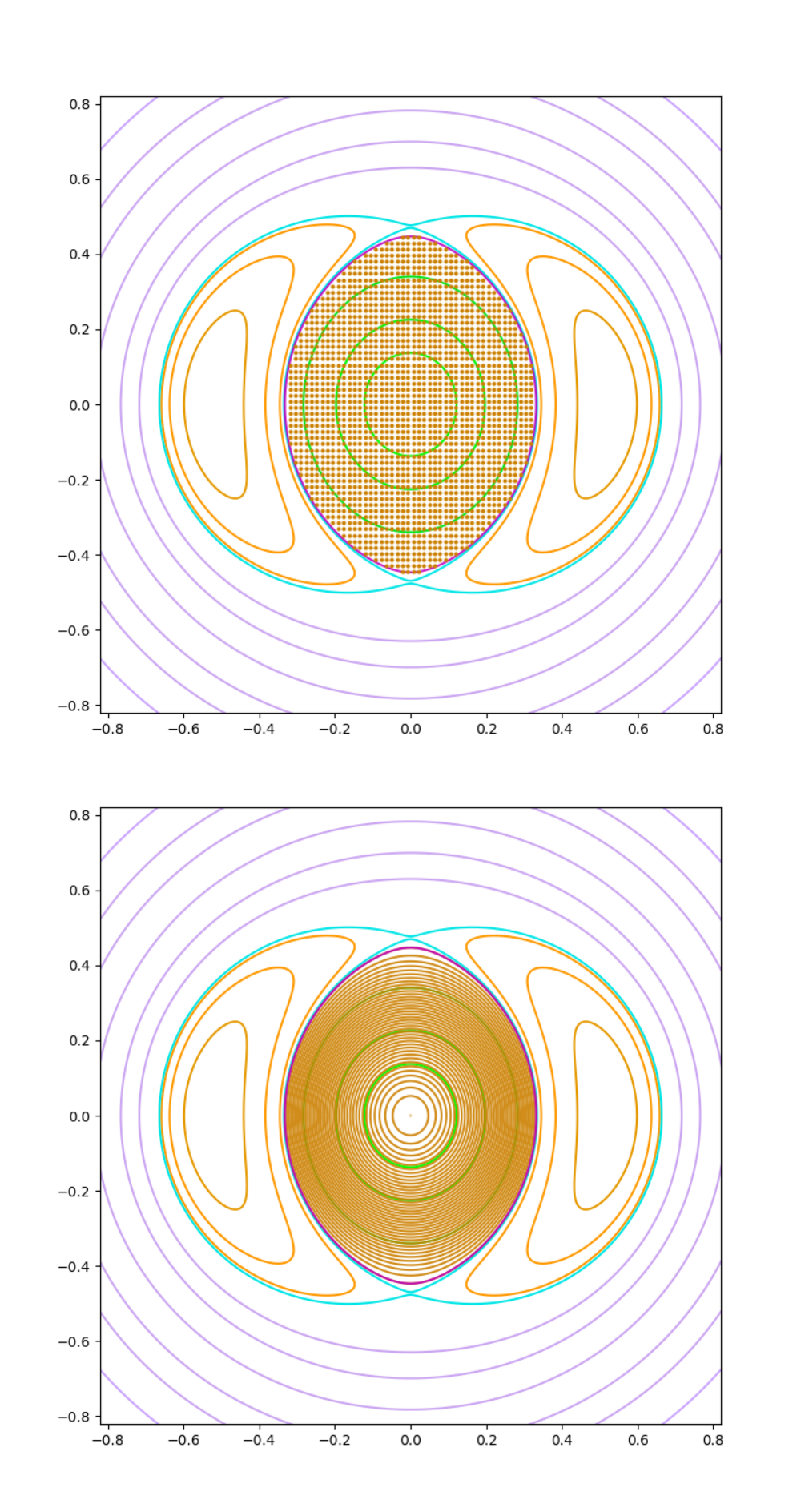}}
   \caption{(Left) {Volume} 
 computed by (\ref{V_general}) (blue), (\ref{eq:vol_lambda}) (black), Theorem \ref{thm4} (red) and Theorem \ref{thm3'} (green) as a function of $|\Psi|$ (to make the slope positive) for the field (\ref{A_ex2}) in the inner region.  The vertical dashed line indicate the value of $\Psi$ for the separatrix. (Right) Example of the grid ({top}) and the set of flux surfaces ({bottom}), corresponding to a uniform partition in~$\Psi$, used to compute (\ref{V_general}) and Theorems~\ref{thm3'} and \ref{thm4}, respectively. 
   }
   \label{Fig_V_inner}
\end{figure}
\unskip
\FloatBarrier                   

\subsubsection{Magnetic~Island}
 A typical computation of $T_2(\Psi)$ for the magnetic island region is displayed in Figure~\ref{Fig_T2_isl}. As~mentioned earlier in this section, $T_2$ is recorded for the second intersection between the fieldline and the $u$-line on the $(\phi,\vartheta)$-plane.

\begin{figure}[ht!]
 \centering  
{\includegraphics[height=6.3cm, trim={60 0 20 0},clip]{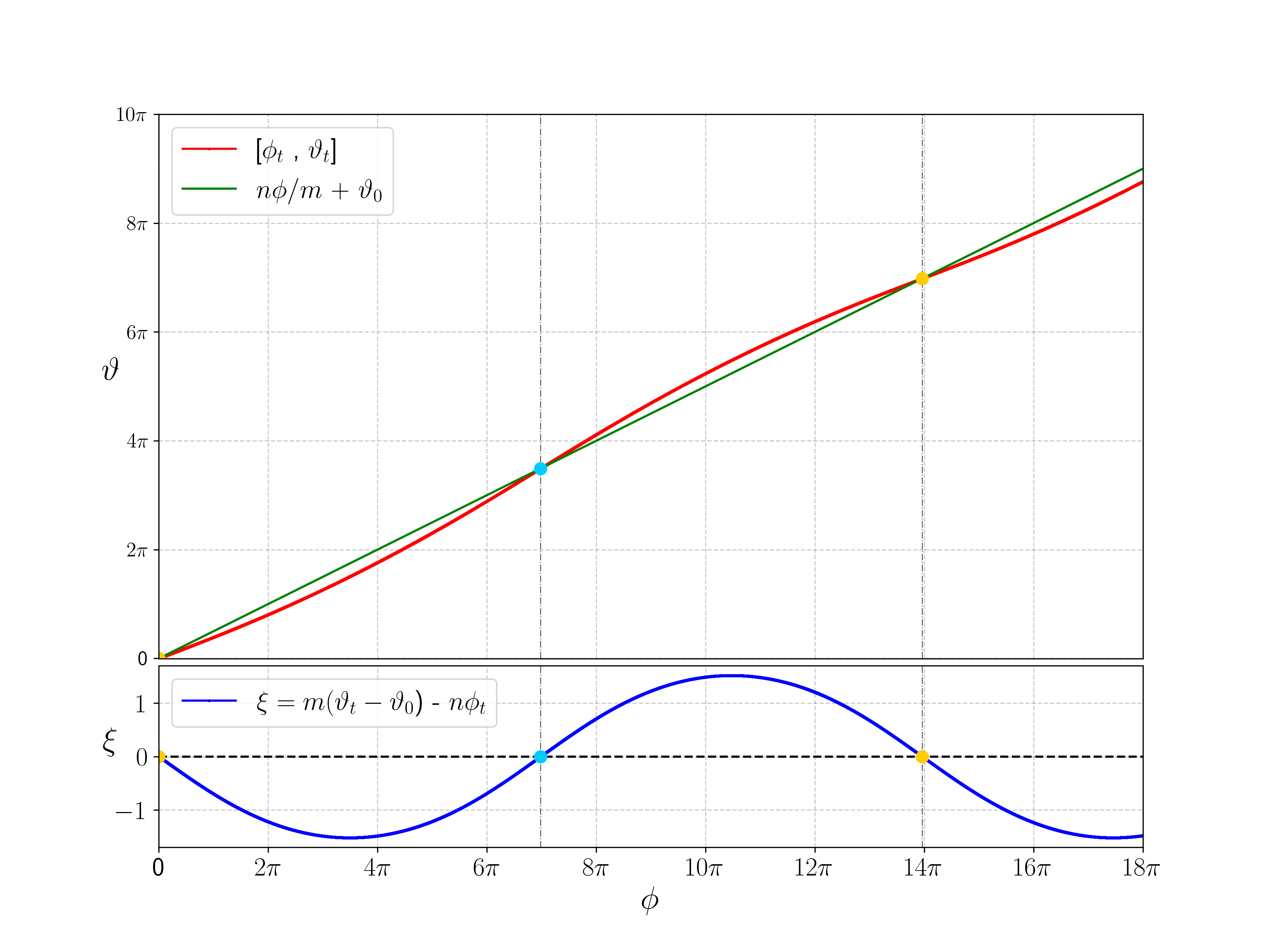}}
{\includegraphics[height=6.3cm, trim={10 10 70 0},clip]{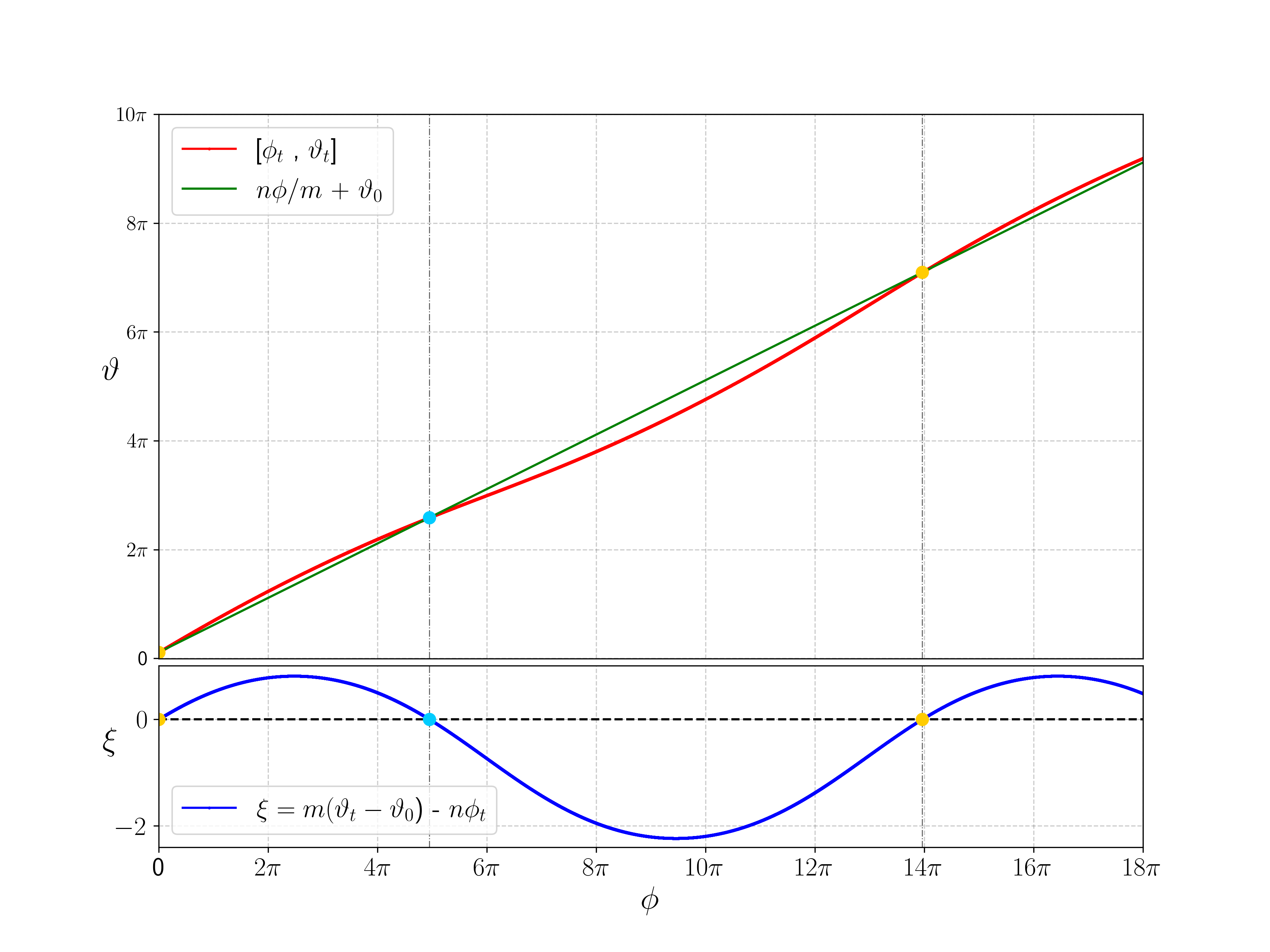}} 
\caption{Crossing 
 of the fieldline (red) and $u$-line (green) starting at two selected initial points $(\tilde{y}_i, \tilde{z}_i)$ on the poloidal plane $\phi=0$: $(0.4,0)$ and $(0.570,0.211)$  on the same flux surface ($\Psi=-0.0316$) and same component, in~the magnetic island region, for~$\varepsilon=0.007$. The~second crossing time under $v$ is the same in both subplots, $T(\Psi) = 43.86$, while the first (spurious crossing time) differs: 21.93 [Left] and 15.53 [Right].
 The blue graph in the bottom subplots corresponds to $\xi = m(\vartheta_t-\vartheta_0) - n\phi_t$, to~more easily identify the crossings between the orbit and the $u$-line.}
   \label{Fig_T2_isl}

\end{figure}

Figure~\ref{Fig_V_isl} shows the volume computed by (\ref{V_general}) and Theorem \ref{thm3} for flux surfaces in the magnetic island as a function of $\Psi$ defined in (\ref{Psi_hellical}).

\begin{figure}[ht!]
 \centering  
{\includegraphics[height=8.0cm, trim={20 0 10 0},clip]{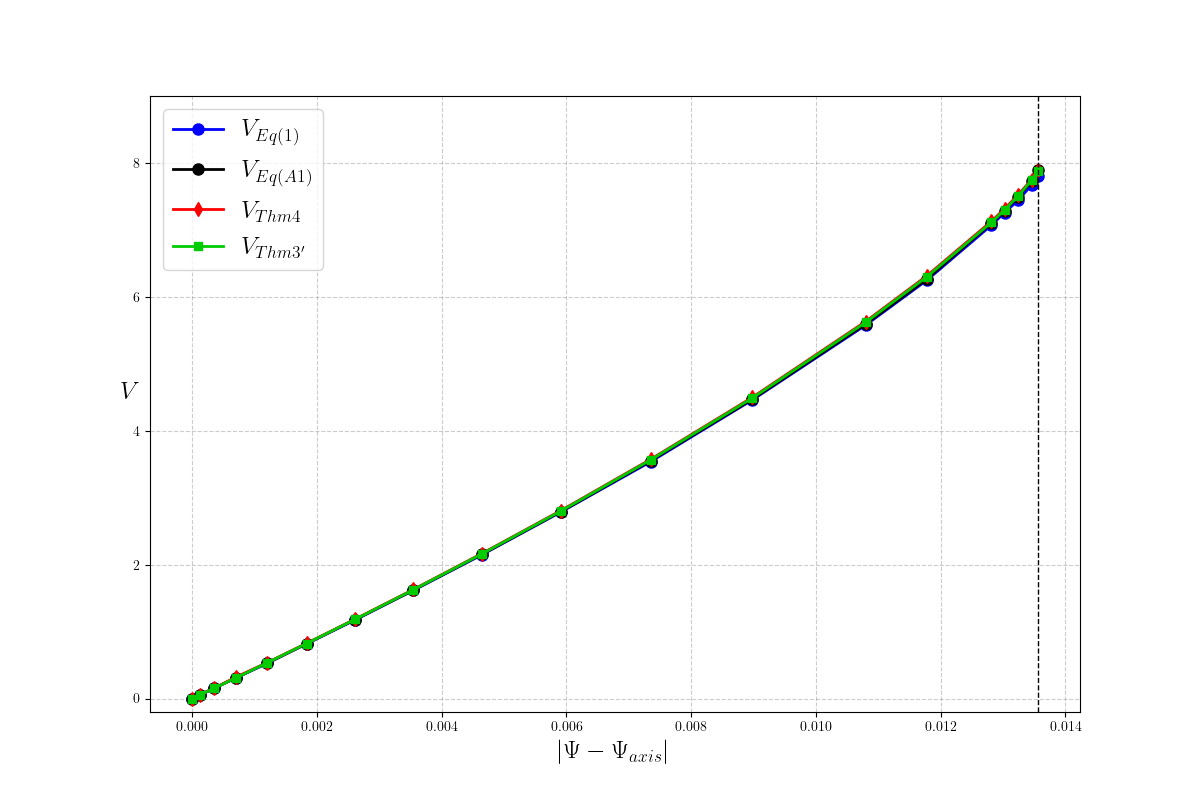}}
{\includegraphics[height=8.0cm, trim={0 0 0 0},clip]{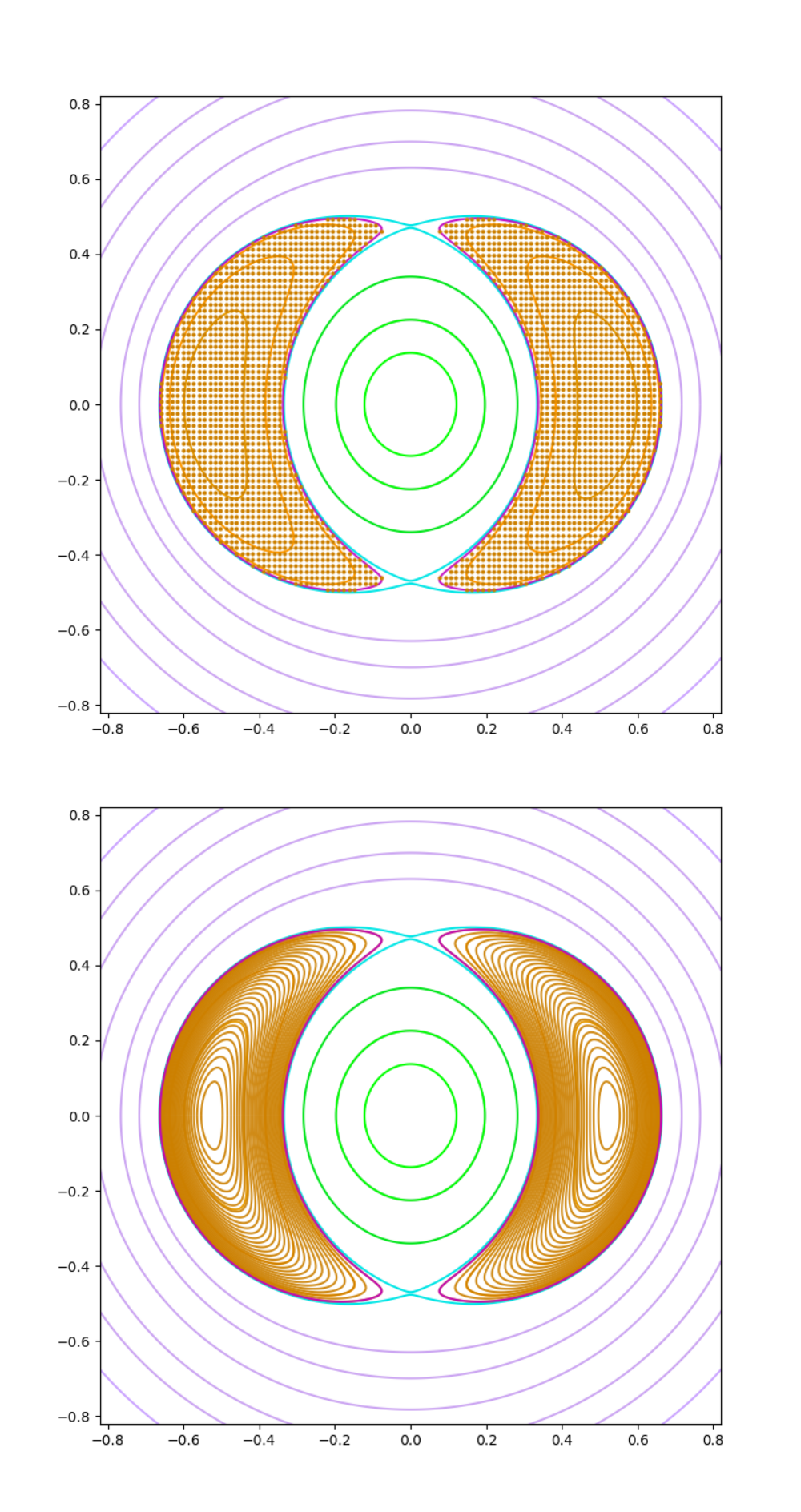}}
   \caption{[Left] {Volume} 
 computed by (\ref{V_general}) (blue), (\ref{eq:vol_lambda}) (black), Theorem \ref{thm4} (red) and Theorem \ref{thm3'} (green) as a function of $\Psi$ for the field (\ref{A_ex2}) in the magnetic island. The vertical dashed line indicate the value of $\Psi$ for the separatrix. 
 [right] Example of the grid ({top}) and the set of flux surfaces ({bottom}), corresponding to a uniform partition in~$\Psi$, used to compute (\ref{V_general}), and~Theorems~\ref{thm3'} and \ref{thm4} and (\ref{eq:vol_lambda}), respectively. 
   }
   \label{Fig_V_isl}

\end{figure}
\unskip
\FloatBarrier                   

\subsubsection{Outer~Region}
 A typical computation of $T_2(\Psi)$ for the outer region is displayed in Figure~\ref{Fig_T2_outer}.

\begin{figure}[ht!]
 \centering  
{\includegraphics[height=6cm, trim={60 0 20 0},clip]{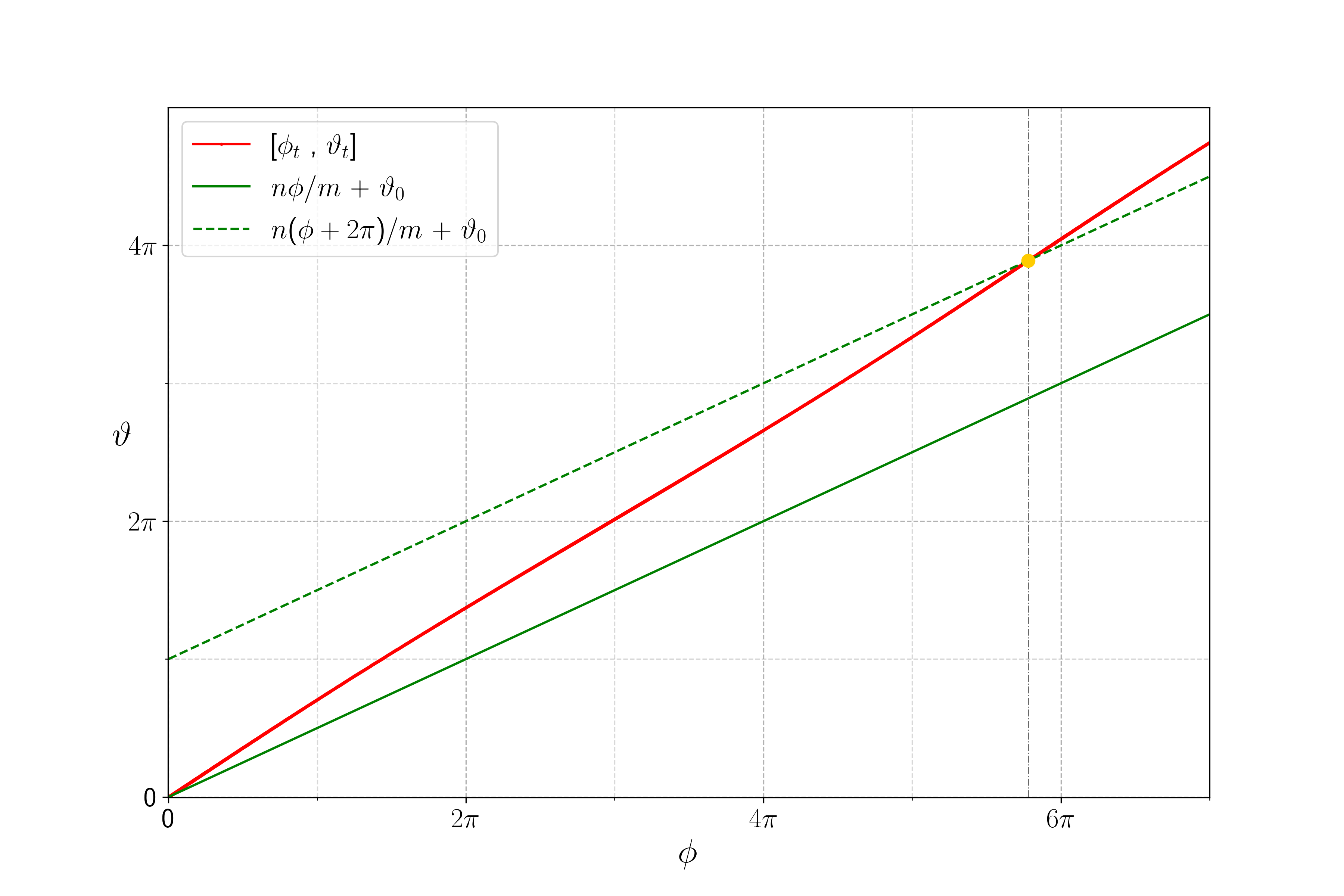}}
{\includegraphics[height=6cm, trim={10 0 70 0},clip]{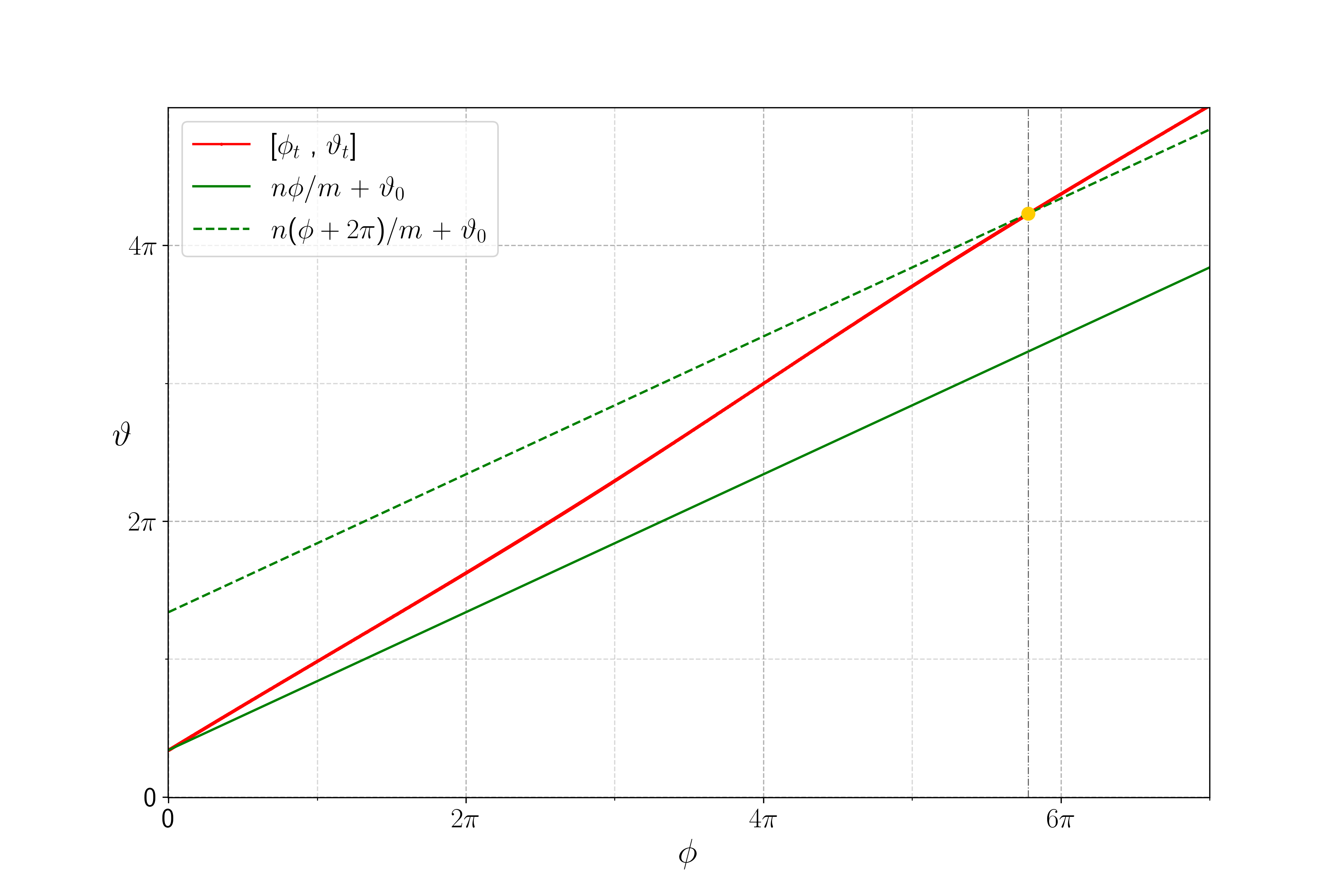}}
\caption{
{Crossing} 
 of the fieldline (red) and images of the $u$-line (green), starting at two selected initial points $(\tilde{y}_i, \tilde{z}_i)$ on the poloidal plane $\phi = 0$: [left] $(0.7, 0)$ and [right] $(0.41, 0.54)$, on~the same flux surface ($\Psi = -0.01533$) in the outer region, for~$\varepsilon = 0.007$.
 The return time is the same in both subplots $T(\Psi)=18.15$.}
   \label{Fig_T2_outer}

\end{figure}

Figure~\ref{Fig_V_outer} shows the volume computed by (\ref{V_general}) and Theorem \ref{thm3} for the inner region as a function of $\Psi$ defined in (\ref{Psi_hellical}).
\begin{figure}[ht!]
 \centering  
{\includegraphics[height=8.0cm, trim={20 0 10 0},clip]{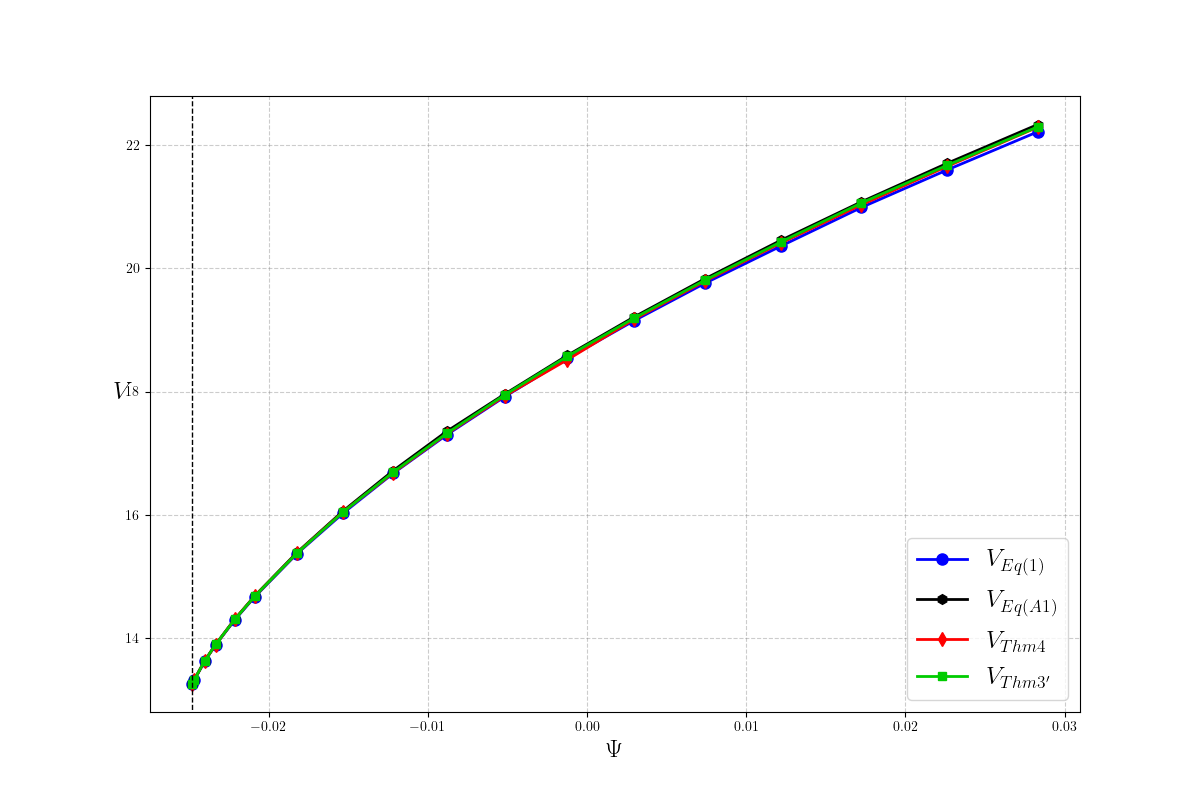}}
{\includegraphics[height=8.0cm, trim={0 0 0 0},clip]{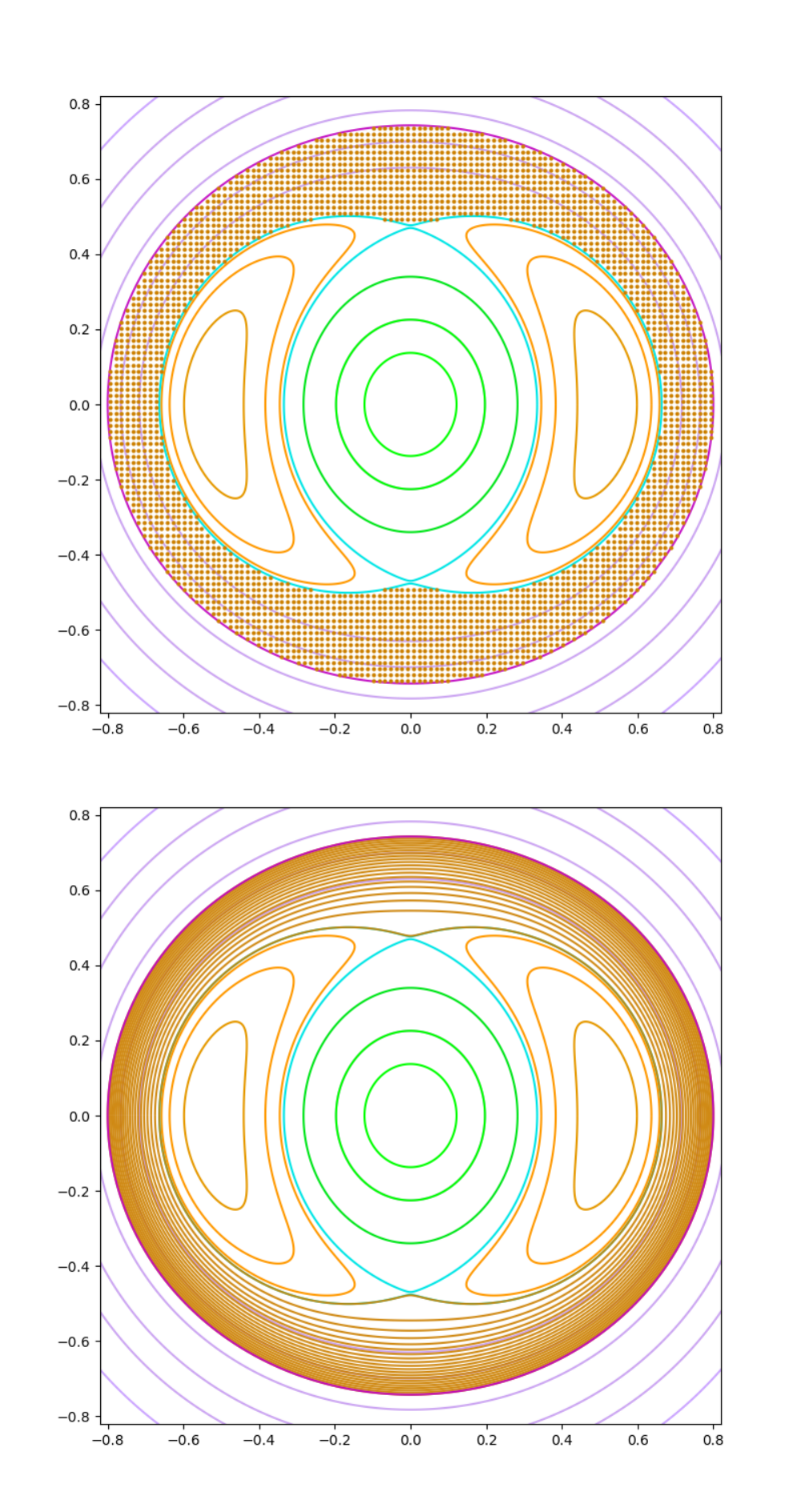}}
   \caption{[Left] {Volume} %
 computed by (\ref{V_general}) (blue), (\ref{eq:vol_lambda}) (black), Theorem \ref{thm4} (red) and Theorem \ref{thm3'} (green) as a function of $\Psi$ for the field (\ref{A_ex2}) in the outer region. The~volumes enclosed by the island and inner regions has been added in. The vertical dashed line indicate the value of $\Psi$ for the separatrix. 
 [Right] Example of the grid ({top}) and the set of flux surfaces ({bottom}), corresponding to a uniform partition in~$\Psi$, used to compute (\ref{V_general}) and Theorems~\ref{thm3'} and \ref{thm4}, respectively.}
   \label{Fig_V_outer}

\end{figure}

\section{Discussion}
\label{sec:discussion}

In this section, we discuss the performance of our implementations of the various methods of computation used in these examples.  There are two aspects:~computation time and accuracy of~the result.

For the axisymmetric field, the~volume computations of (\ref{V_general}) and Theorem \ref{thm1} 
were straightforward and the only parameter values to vary to improve the accuracy of the volume enclosed between a given $\Psi_0$ and $\Psi_1$, were the grid size for each method, as~the integration times required to compute the respective return times are 
of the same order.  
In general, we found that Theorem \ref{thm1} yields higher accuracy in  only a fraction of the runtime of (\ref{V_general}) {or (\ref{eq:vol_lambda})}, as~illustrated in Table~\ref{Ex1_table}, which reports average runtimes for four values of $\Psi$ from Figure~\ref{Fig_V_axissym}, with~grids chosen to achieve approximately 1\% relative error for (\ref{V_general}) and a coarse grid with $N = 20$ for Theorem \ref{thm1}. Our implementation of (\ref{eq:vol_lambda}) was consistently faster and more accurate than (\ref{V_general}), but~not as efficient as Theorem \ref{thm1}.
The average runtimes are given in seconds and correspond to single-core computations on an AMD\textsuperscript{\textregistered} Ryzen~5~3500U (2.1~GHz).

\vspace{0.2cm}
\begin{table}[h!]
\centering
\begin{tabular}{cc cc cc c c c}
\specialrule{.12em}{0em}{0em}
  {$R_1$} & {$\Psi$}  & {$V_e$} 
  & Method & {$N_1$}& {$N_2$} & {$V_0$}
    &  {$\langle t_0\rangle$}\textbf{(s)} & {$\Delta/V_e$}
    \\
\specialrule{.12em}{0em}{0em}
 \multirow{4}{*}{$1.2$} & \multirow{4}{*}{0.020}   
 & \multirow{4}{*}{0.789568}   
       & Eq~(\ref{V_general}) & 300 &300  & 0.782214 & 67 & 0.009 
       \\ \cline{4-9}
    &&& \cellcolor{grey3}Eq~(\ref{eq:vol_lambda}) &   \cellcolor{grey3}50 & \cellcolor{grey3}50
      & \cellcolor{grey3}0.787261  &   \cellcolor{grey3}56 
      &  \cellcolor{grey3}0.003 \\ \cline{4-9}
   &&& \cellcolor{grey2}Thm \ref{thm1} & \cellcolor{grey2}20  & \cellcolor{grey2}---
      &  \cellcolor{grey2}0.789548 & \cellcolor{grey2}5   
      & \cellcolor{grey2}$2.6\times 10^{-5}$ \\
\specialrule{.10em}{-.1em}{0em}
 \multirow{4}{*}{$1.4$} & \multirow{4}{*}{0.080}   
 & \multirow{4}{*}{3.158273}   
       & Eq~(\ref{V_general}) & 200 & 200 & 3.123140 & 107 & 0.011 \\ \cline{4-9}
    &&& \cellcolor{grey3}Eq~(\ref{eq:vol_lambda}) &   \cellcolor{grey3}50& \cellcolor{grey3}50
      &  \cellcolor{grey3}3.149079 &   \cellcolor{grey3}57
      &  \cellcolor{grey3}0.003 \\ \cline{4-9}
    &&& \cellcolor{grey2}Thm \ref{thm1} & \cellcolor{grey2}20  & \cellcolor{grey2}---
      &  \cellcolor{grey2}3.158248 & \cellcolor{grey2}5   
      & \cellcolor{grey2}$7.9\times 10^{-6}$  
    \\ \specialrule{.10em}{-.1em}{0em}
 \multirow{4}{*}{$1.6$} & \multirow{4}{*}{0.180}   
 & \multirow{4}{*}{7.106115}   
       & Eq~(\ref{V_general}) & 200 & 200 & 7.016993& 240 & 0.012 
    \\ \cline{4-9}
    &&& \cellcolor{grey3}Eq~(\ref{eq:vol_lambda}) &   \cellcolor{grey3}50 & \cellcolor{grey3}50
      & \cellcolor{grey3}7.085340  &   \cellcolor{grey3}57 
      & \cellcolor{grey3}0.003 \\ \cline{4-9}
    &&& \cellcolor{grey2}Thm \ref{thm1} & \cellcolor{grey2}20  & \cellcolor{grey2}---
      &  \cellcolor{grey2}7.106088 & \cellcolor{grey2}4   
      & \cellcolor{grey2}$3.8\times 10^{-6}$  
    \\ \specialrule{.10em}{-.1em}{0em}
 \multirow{4}{*}{$1.8$} & \multirow{4}{*}{0.320}   
 & \multirow{4}{*}{12.633094}   
       & Eq~(\ref{V_general}) & 200 & 200  & 12.488658 & 403 & 0.011 
    \\ \cline{4-9}
    &&& \cellcolor{grey3}Eq~(\ref{eq:vol_lambda}) &   \cellcolor{grey3}50 & \cellcolor{grey3}50
      & \cellcolor{grey3}12.595751  &   \cellcolor{grey3}58
      & \cellcolor{grey3}0.003 \\ \cline{4-9}
    &&& \cellcolor{grey2}Thm \ref{thm1} & \cellcolor{grey2}20  & \cellcolor{grey2}---
      &  \cellcolor{grey2}12.633059 & \cellcolor{grey2}4  
      & \cellcolor{grey2}$2.7\times 10^{-6}$ \\  
\specialrule{.12em}{-.1em}{0em}
\end{tabular}
\caption{Comparison of the volume estimation $V_0$, enclosed by the flux surfaces passing through $[1,0]$ ($\Psi_0 = 0$) and $[R_1,0]$ ($\Psi_1$), obtained using (\ref{V_general}) and Theorem \ref{thm1}, with~respect to grid size and average run time. The~grid dimensions are $N_1 \times N_2$ for (\ref{V_general}), $N_g = N_2$ for the discretised level sets in (\ref{eq:vol_lambda}), and~$N_{\Psi} = N_1$ for integration in $\Psi$ in (\ref{eq:vol_lambda}) and Theorems~\ref{thm3'} and~\ref{thm4}. The~rectangular domain for~(\ref{V_general}) is $[0.1,1.9]\times[-0.9,0.9]$. \label{Ex1_table}}
\end{table}

For the toroidal helical field, the~computations of (\ref{V_general}) and Theorems \ref{thm3'} and \ref{thm4} yielded comparable results but with significantly different runtimes.
Accuracy was controlled by the grid sizes: $N \times N$ points in the poloidal section for (\ref{V_general}), $N_g$ points to discretise a level set and $N$ on $\Psi$ for (\ref{eq:vol_lambda}), and~ $N$ points in $\Psi$ for Theorems \ref{thm3'} and \ref{thm4}.  We quantified accuracy by the relative error compared to a reference value $V^*$.
The integration times for (\ref{V_general}) and Theorem \ref{thm3'}  were of the same order, whereas that of Theorem \ref{thm4} was required to be at least one order of magnitude greater to estimate the average return time. 
As (\ref{V_general}) is equivalent to a three-dimensional integration, whereas (\ref{eq:vol_lambda}) and Theorems~\ref{thm3'} and~\ref{thm4} are equivalent to a bit more than two dimensions, one would expect the codes based on those theorems or (\ref{eq:vol_lambda}) to run faster than that for (\ref{V_general}). In~practice, however, for~the small volumes between close flux surfaces ($|\Psi_1-\Psi_2|< \epsilon$), our use of Theorem~\ref{thm4} can be slower than (\ref{V_general}), and~(\ref{eq:vol_lambda}) is the slowest of the four. There is a tradeoff between accuracy and computation time between our implementations of (\ref{V_general}) and (\ref{eq:vol_lambda}), with~the first being faster in small volumes and the second always more accurate, with a consistent computation time. 
Our implementation of (\ref{V_general}) computes $\mathsf{T}$ only for the points in the grid contained between the level sets of $\Psi_1$ and $\Psi_2$, while that of (\ref{eq:vol_lambda}) computes $\mathsf{T}$ for every point in the discretised contour for $N$ contours in the partition of $\Psi$. 
Table~\ref{Ex2_table} illustrates performance by reporting average runtimes for six $[\Psi_1,\Psi_2]$ intervals in the inner, island, and~outer regions, using typical grid sizes. The~reference values of~$\Psi_1$ correspond to the magnetic axis ($\Psi = 0$), the~island O-point ($\Psi = -0.0384$), and~the separatrix ($\Psi = -0.0248$). The~reference volume $V_*$, used to compute the relative error, is obtained from Theorem~\ref{thm3'} with $N=400$, as~this proved the most accurate and efficient method. More comparable implementations would allow more meaningful comparisons to be~made.

In conclusion, our tests on Theorem \ref{thm4} sometimes did not give efficient volume computations, but~those on Theorem \ref{thm3'} generally did.  As~most integrable cases can also be treated by Theorem \ref{thm3'} after identification of a symmetry that conserves a density, our recommendation is to use Theorem \ref{thm3'} (unless Theorems \ref{thm2} or \ref{thm1} apply, in~which case the need to compute the harmonic average of the density is removed).

\begin{table}[h!] 
\footnotesize
\begin{tabular}{ccc c cc c c  c  c c} 
\specialrule{.12em}{0em}{0em}
 \boldmath{$\tilde{y}_1$} & \boldmath{$\Psi_1$} & \boldmath{$\tilde{y}_2$}   & \boldmath{$\Psi_2$} & \textbf{Region} 
  & \textbf{Method} & \boldmath{$N_1$} &  {$N_2$} &\boldmath{$V_0$}
    &  \boldmath{$\langle t_0\rangle$}\textbf{(s)} & \boldmath{$\Delta/V_{*}$}
    \\   \specialrule{.10em}{0em}{0em}
 \multirow{4}{*}{$0$} & \multirow{4}{*}{0.0} &
 \multirow{4}{*}{$0.150$} & \multirow{4}{*}{$-$0.0060}
 & \multirow{4}{*}{Inner}   
    & Eq~(\ref{V_general}) & 300 & 300  & 0.997447 & 145 & 0.006 
       \\ \cline{6-11}
&&&&&  \cellcolor{grey3}Eq~(\ref{eq:vol_lambda}) & \cellcolor{grey3}100  
      & \cellcolor{grey3}100
      & \cellcolor{grey3}1.002811
      & \cellcolor{grey3}261
      & \cellcolor{grey3}$7.7\times 10^{-4}$
       \\ \cline{6-11}
&&&&& Thm \ref{thm3'} & 100  & ---
      &  1.013366 & 31   
      &  $9.7\times 10^{-4}$  
      \\ \cline{6-11}
&&&&& \cellcolor{grey3}Thm \ref{thm4} 
      & \cellcolor{grey3}100 
      & \cellcolor{grey3}---
      & \cellcolor{grey3}1.003619 & \cellcolor{grey3}115   
      & \cellcolor{grey3}$3.3\times 10^{-5}$  
    \\ \specialrule{.10em}{-.1em}{0em}
 \multirow{4}{*}{$0$} & \multirow{4}{*}{0.0}  &
 \multirow{4}{*}{$0.320$} & \multirow{4}{*}{$-$0.0245}
 & \multirow{4}{*}{Inner}   
       & Eq~(\ref{V_general}) & 300 & 300  & 4.998763 & 652 & 0.005 
       \\ \cline{6-11}
 &&&&&  \cellcolor{grey3}Eq~(\ref{eq:vol_lambda}) 
      & \cellcolor{grey3}100  
      & \cellcolor{grey3}100 
      & \cellcolor{grey3}5.027210
      & \cellcolor{grey3}252
      & \cellcolor{grey3}$8.1\times 10^{-4}$  
       \\ \cline{6-11}
 &&&&& Thm \ref{thm3'} & 100  & ---
      & 5.031270 
      & 30
      & 0.002 
      \\ \cline{6-11}
 &&&&& \cellcolor{grey3}Thm \ref{thm4} 
      & \cellcolor{grey3}100  
      & \cellcolor{grey3}--- 
      & \cellcolor{grey3}5.031856 
      & \cellcolor{grey3}114   
      & \cellcolor{grey3}0.002  
    \\ \specialrule{.10em}{-.1em}{0em}
 \multirow{4}{*}{$0.52542$} & \multirow{4}{*}{$-$0.0384}  &
 \multirow{4}{*}{$0.550$} & \multirow{4}{*}{$-$0.0380}
 & \multirow{4}{*}{Island}   
       & Eq~(\ref{V_general}) & 500  & 500
       & 0.155675 
       & 17 
       & 0.003 
       \\ \cline{6-11}
&&&&&  \cellcolor{grey3}Eq~(\ref{eq:vol_lambda}) & \cellcolor{grey3}100  
      & \cellcolor{grey3}100
      & \cellcolor{grey3}0.154939
      & \cellcolor{grey3}310  
      & \cellcolor{grey3}0.002 
       \\ \cline{6-11}
&&&&& Thm \ref{thm3'} & 100  & ---
      &  0.155206 & 16   
      &  $5.1\times 10^{-5}$ 
      \\ \cline{6-11}
&&&&& \cellcolor{grey3}Thm \ref{thm4} & \cellcolor{grey3}100  
      &  \cellcolor{grey3}---
      & \cellcolor{grey3}0.155038
      & \cellcolor{grey3}58   
      & \cellcolor{grey3}0.001  
    \\  \specialrule{.10em}{-.1em}{0em}
 \multirow{4}{*}{$0.52542$} & \multirow{4}{*}{$-$0.0384}  &
 \multirow{4}{*}{$0.662$} & \multirow{4}{*}{$-$0.0251}  
 & \multirow{4}{*}{Island}   
       & Eq~(\ref{V_general}) & 500 &500 & 7.470030 & 640 & 0.005 
       \\ \cline{6-11}
&&&&&  \cellcolor{grey3}Eq~(\ref{eq:vol_lambda}) 
      & \cellcolor{grey3}100  
      & \cellcolor{grey3}150
      & \cellcolor{grey3}7.492240
      & \cellcolor{grey3}623  
      & \cellcolor{grey3}0.002 
       \\ \cline{6-11}
 &&&&& Thm \ref{thm3'} & 100  & ---
      &  7.521046
      & 18   
      & 0.002
      \\ \cline{6-11}
 &&&&& \cellcolor{grey3}Thm \ref{thm4} 
     & \cellcolor{grey3}100  & \cellcolor{grey3}---
 & \cellcolor{grey3}7.519672 
      &\cellcolor{grey3}106   
      & \cellcolor{grey3}0.002 
    \\ \specialrule{.10em}{-.1em}{0em}
 \multirow{4}{*}{$0.66345$} & \multirow{4}{*}{$-$0.0248}  &
 \multirow{4}{*}{$0.670$} & \multirow{4}{*}{$-$0.0240}   
 & \multirow{4}{*}{Outer}   
       & Eq~(\ref{V_general}) & 600  &600 & 0.639184 & 58 & 0.014
       \\ \cline{6-11}
&&&&&  \cellcolor{grey3}Eq~(\ref{eq:vol_lambda}) 
      & \cellcolor{grey3}100  & \cellcolor{grey3}100 
      & \cellcolor{grey3}0.650239
      & \cellcolor{grey3}361
      & \cellcolor{grey3}0.003  
       \\ \cline{6-11}
 &&&&& Thm \ref{thm3'} & 100  & ---
      &  0.649799
      &   44
      &  0.002   
      \\ \cline{6-11}
 &&&&& \cellcolor{grey3}Thm \ref{thm4} 
      & \cellcolor{grey3}100 
      &  \cellcolor{grey3}---
      & \cellcolor{grey3}0.650996 
      & \cellcolor{grey3}212  
      & \cellcolor{grey3}0.004   
    \\  \specialrule{.10em}{-.1em}{0em}
 \multirow{4}{*}{$0.66345$} & \multirow{4}{*}{$-$0.0248}  &
 \multirow{4}{*}{$0.780$} & \multirow{4}{*}{0.0172}       
 & \multirow{4}{*}{Outer}   
       & Eq~(\ref{V_general}) & 600 &600 & 7.736527 & 614 & 0.008 
       \\ \cline{6-11}
&&&&&  \cellcolor{grey3}Eq~(\ref{eq:vol_lambda}) 
      & \cellcolor{grey3}100 & \cellcolor{grey3}100
      & \cellcolor{grey3}7.906824
      & \cellcolor{grey3}354
      & \cellcolor{grey3}0.014
       \\ \cline{6-11}
 &&&&& Thm \ref{thm3'} & 100  & ---
      & 7.904392 
      & 44   
      &  0.014  
      \\ \cline{6-11}
 &&&&& \cellcolor{grey3}Thm \ref{thm4} 
      & \cellcolor{grey3}100  &  \cellcolor{grey3}--- 
      & \cellcolor{grey3}7.909134
      & \cellcolor{grey3}173   
      & \cellcolor{grey3}0.014
    \\
\specialrule{.12em}{-.1em}{0em}
\end{tabular}
\caption{Comparison of the volume estimations $V_0$ enclosed by the flux surfaces $\Psi_1$ and $\Psi_2$, passing through $(\tilde{y}_1,0)$ and $(\tilde{y}_2,0)$, respectively. The~estimates are obtained from~(\ref{V_general}), {(\ref{eq:vol_lambda})} and Theorems~\ref{thm3'} and~\ref{thm4}, with~results shown against grid size, average runtime, and~relative error with respect to a reference volume $V_*$. 
The grid dimensions are $N_1 \times N_2$ for (\ref{V_general}), $N_g = N_2$ for the discretised level sets in (\ref{eq:vol_lambda}), and~$N_{\Psi} = N_1$ for integration in $\Psi$ in (\ref{eq:vol_lambda}) and Theorems~\ref{thm3'} and~\ref{thm4}. The~rectangular domains for~(\ref{V_general}) are inner region $[-0.335,\,0.335] \times [-0.47,\,0.47]$; island and outer regions $[-0.9,\,0.9] \times [-0.8,\,0.8]$.
The reference volume $V_*$ is given by Theorem~\ref{thm3'} with $N=400$. \label{Ex2_table}}
\end{table}

\section{Conclusions}
\label{sec:conc}

We computed the volumes enclosed by flux surfaces for two different integrable examples of magnetic fields using  formulae provided by~\cite{mackay2024volume} and a new formula presented here for integrable fields with a symmetry field that preserves a density. The~results of using different formulae are found to be consistent up to numerical error. As~the methods are equivalent to 2D integration, they are mostly faster to compute than the standard 3D~integration. 

The volume estimation for the axisymmetric field between codes based on (\ref{V_general}), (\ref{eq:vol_lambda}) and Theorem \ref{thm1} confirmed that the latter was notably more efficient and  accurate. This is evident in the comparison between a coarse grid in $\Psi$ for Theorem \ref{thm1} and a finer 2D grid for (\ref{V_general}) and (\ref{eq:vol_lambda}) included in the previous section.
Regarding the toroidal helical field, among~the three formulae tested, the~computational cost of Theorem~\ref{thm3'} was distinctly lower than that of (\ref{V_general}) and Theorem~\ref{thm4}, and~exhibited faster numerical convergence. The~accuracy of the integration in (\ref{V_general}) was constrained by the resolution of the regular grid employed for its evaluation. Nevertheless, for~flux surfaces with a small cross-section, (\ref{V_general}) is more efficient than Theorem~\ref{thm4}, since accurate integration along the flow lines was substantially more demanding, with~lengths greater by approximately an order of magnitude (i.e., by~a factor of ten).

A disadvantage of these methods is that they cannot be used to measure volumes enclosed by flux surfaces in fields that contain chaotic regions.  
It would be good to obtain an efficient formula for the flux enclosed by a flux surface of a general magnetic field.  We plan to achieve and test this via a vector potential on the flux surface, which easily gives efficient formulae for enclosed toroidal and poloidal flux, and~we have just worked out how to get it to give~volume.

It would be desirable to test the formulae on numerically computed quasisymmetric fields as in~\cite{LP22}.
The methods should also apply to approximately integrable fields with knotted or linked tori, as~in Ref.~\cite{HSF}, though~it is difficult to make exactly integrable examples (indeed, we have conjectured many years ago that non-degenerate integrable vacuum fields have to be axisymmetric).

\section*{Acknowledgements}
This work was supported by a grant from the Simons Foundation (601970, RSM).

\subsection*{Code availability}
Code for the computations is available on \cite{Vformulae2025b}.

\section*{Appendix A. Computation of (\ref{V_general}) for Integrable Fields}
\label{sec:app}
\renewcommand{\theequation}{A\arabic{equation}}
\setcounter{equation}{0}

For an integrable field $B$, if~one wants to use (\ref{V_general}) to compute $V(\Psi)$ for a range of $\Psi$, one can evaluate the integrand at a large enough grid to cover the largest domain (say the largest value of $\Psi$ assuming $\Psi$ increases outwards), and~label each evaluation by its value of $\Psi$.  Then, for a given value of $\Psi$, it suffices to add up the results for smaller values of $\Psi$.  This can be done progressively as $\Psi$ increases.

It might, however, be more efficient (for given accuracy) to compute
\begin{equation}
\label{eq:vol_lambda}
\frac{dV}{d\Psi}(\Psi) = \int_{\partial D(\Psi)} \mathsf{T} \lambda \,,
\end{equation}
where $\lambda = i_Bi_n\Omega$ and $n = \nabla \Psi/|\nabla \Psi|^2$ (with respect to any Riemannian metric), and~then take its definite integral with respect to $\Psi$ from a reference case (e.g.,~magnetic axis).  To~prove  (\ref{eq:vol_lambda}), first note that $\beta = \lambda \wedge d\Psi$.  
This is because applying either side to $n$ gives $-\lambda$, and~applying either side to $B$ gives $0$.  Lastly, for~any $\xi$ independent from $n$ and $B$, evaluate both sides on $(n,\xi)$ and $(B,\xi)$ and obtain $-i_\xi \lambda$ and $0$, respectively;~evaluating on $(\xi,\xi)$ gives zero automatically on both sides. Then integrating $\mathsf{T} \lambda \wedge d\Psi$ over $\partial D(\Psi)$ gives the~result.

To compute the integral (\ref{eq:vol_lambda}) numerically, one could discretise $\partial D(\Psi)$ uniformly with respect to some choice $\nu$ of angle variable (mod 1) around it, $\nu_n = n/N$. Let $\epsilon_n$ be half the displacement vector from $\nu_{n-1}$ to $\nu_{n+1}$, evaluate $\lambda$ on $\epsilon_n$ considered to be based at $\nu_n$, and~compute their sum.  
Note that $i_\epsilon\lambda = \Omega(n,B,\epsilon)$.  
In vector calculus language, this is given by the triple product $i_{\epsilon}\lambda = n \times B \cdot \epsilon$, but~if the field is presented in non-trivial coordinates $x^i$, then it is simpler to use $i_{\epsilon} \lambda = \sqrt{|g|} \det W$, where $\sqrt{|g|}$ is the volume factor for the metric and $W$ is the matrix formed by the contravariant components of the three vectors in that coordinate system.  A~further simplification can be made by choosing the Riemannian metric in the definition of $n$ to be the standard one for the coordinate system, so $\nabla \Psi$ has components $\frac{\partial \Psi}{\partial x^i}$ (all that is required of $n$ is that $i_n d\Psi = 1$).
If everything is analytic, then by Fourier analysis, the error will be exponentially small in $N$.

Computing (\ref{eq:vol_lambda}) is a 2D integration ($\mathsf{T}$ is a 1D integral and it is integrated along $\partial D(\Psi)$), but~because one can get away with relatively coarse discretisation, we can consider it to be just a bit more than 1D.  So after integrating with respect to $\Psi$, the total is a bit more than a 2D~integration.

\bibliographystyle{unsrt}
\bibliography{Volumenbib}

\end{document}